%% file: bdim-ldim.tex
\documentclass[11pt]{amsart}

\usepackage[utf8]{inputenc}
\usepackage[T1]{fontenc}
\usepackage{amssymb,graphicx}
\usepackage{showlabels} 
\usepackage{enumitem}
\usepackage{expl3}
\usepackage{xparse}
\usepackage{amsthm}
\usepackage{hyperref}
\usepackage[margin=1.5in]{geometry}
\usepackage{tikz}
\usetikzlibrary{positioning,calc,shapes.geometric}
\usepackage{xifthen}

\usepackage[pagewise]{lineno}
\usepackage{todonotes}
\newcommand*\patchAmsMathEnvironmentForLineno[1]{%
  \expandafter\let\csname old#1\expandafter\endcsname\csname #1\endcsname
  \expandafter\let\csname oldend#1\expandafter\endcsname\csname end#1\endcsname
  \renewenvironment{#1}%
     {\linenomath\csname old#1\endcsname}%
     {\csname oldend#1\endcsname\endlinenomath}}%
\newcommand*\patchBothAmsMathEnvironmentsForLineno[1]{%
  \patchAmsMathEnvironmentForLineno{#1}%
  \patchAmsMathEnvironmentForLineno{#1*}}%
\AtBeginDocument{%
\patchBothAmsMathEnvironmentsForLineno{equation}%
\patchBothAmsMathEnvironmentsForLineno{align}%
\patchBothAmsMathEnvironmentsForLineno{flalign}%
\patchBothAmsMathEnvironmentsForLineno{alignat}%
\patchBothAmsMathEnvironmentsForLineno{gather}%
\patchBothAmsMathEnvironmentsForLineno{multline}%
}

\ExplSyntaxOn
\DeclareDocumentEnvironment{repetition}{m o}
{
  \tl_set_eq:Nc \l_tmpa_tl {#1}
  \IfNoValueTF{#2}{
    \tl_set:Nx \l_tmpb_tl {
      \tl_item:Nn \l_tmpa_tl {1}
      {\tl_item:Nn \l_tmpa_tl {2} }
      {}
      {\tl_item:Nn \l_tmpa_tl {4}}
    }
  }
  {
    \tl_set:Nx \l_tmpb_tl {
      \tl_item:Nn \l_tmpa_tl {1}
      {\tl_item:Nn \l_tmpa_tl {2} }
      {}
      {\tl_item:Nn \l_tmpa_tl {4} ~ \exp_not:N \ref{#2} }
    }
  }
  \tl_use:N \l_tmpb_tl
}
{
  \use:c {end #1}
}
\ExplSyntaxOff

\newtheorem{theorem}{Theorem}[section]

\newtheorem{lemma}[theorem]{Lemma}

\theoremstyle{remark}

\newcommand{\bbP}{\mathbb{P}}
\newcommand{\bfk}{\mathbf{k}}
\newcommand{\bfn}{\mathbf{n}}
\newcommand{\bfh}{\mathbf{h}}
\newcommand{\bfII}{\mathbf{2}}
\newcommand{\bfV}{\mathbf{5}}

\newcommand{\cgB}{\mathcal{B}}

\newcommand{\cgL}{\mathcal{L}}
\newcommand{\cgM}{\mathcal{M}}

\newcommand{\cgP}{\mathcal{P}}
\newcommand{\cgR}{\mathcal{R}}

\newcommand{\cgU}{\mathcal{U}}
\newcommand{\Inc}{\operatorname{Inc}}
\newcommand{\Int}{\operatorname{Int}}
\newcommand{\Max}{\operatorname{Max}}
\newcommand{\Min}{\operatorname{Min}}
\newcommand{\ldim}{\operatorname{ldim}}
\newcommand{\bdim}{\operatorname{bdim}}
\newcommand{\st}{\operatorname{stack}}
\newcommand{\ple}{\operatorname{ple}}
\newcommand{\Ram}{\operatorname{Ram}}
\newcommand{\PRam}{\operatorname{PRam}}
\newcommand{\BTRam}{\operatorname{BTRam}}
\newcommand{\tw}{\operatorname{tw}}
\newcommand{\pw}{\operatorname{pw}}
\newcommand{\dist}{\operatorname{dist}}

\let\emptyset\varnothing

\begin{document}
\reversemarginpar

\title[Comparing Dimension and Two Variations]{Comparing Dushnik-Miller Dimension,\\
  Boolean Dimension and Local Dimension}

\author[Barrera-Cruz]{{Fidel Barrera-Cruz$^*$}}
\address[Barrera-Cruz]{Sunnyvale, CA}
\email{fidel.barrera@gmail.com}
  
 \author[PRAG]{{Thomas Prag}}
\address[Prag, Trotter]{School of Mathematics, Georgia Institute of Technology, Atlanta, Georgia 30332}
\email{\{tprag3,trotter\}@math.gatech.edu}
  
 \author[Smith]{{Heather C. Smith$^*$}}
\address[Smith]{Department of Mathematics and Computer Science, Davidson College, Davidson, NC 28035}
\email{hcsmith@davidson.edu} 

\author[TAYLOR]{Libby Taylor$^*$}
\address[Taylor]{Department of Mathematics\\
Stanford University\\
Stanford, CA 94305}
\email{libbytaylor@stanford.edu}

\author[TROTTER]{William T. Trotter}
  
\thanks{$^*$Much of the research was completed while affiliated with the Georgia Instutite of Technology.}  
  
\date{June 19, 2019}



\begin{abstract}
  The original notion of dimension for posets is due to Dushnik and
  Miller and has been studied extensively in the literature.  Quite
  recently, there has been considerable interest in two variations of
  dimension known as \textit{Boolean dimension} and \textit{local
    dimension}.  For a poset $P$, the Boolean dimension of $P$ and the
  local dimension of $P$ are both bounded from above by the dimension
  of $P$ and can be considerably less.  Our primary goal will be to
  study analogies and contrasts among these three parameters.  As one
  example, it is known that the dimension of a poset is bounded as a function of its
  height and the tree-width of its cover graph. The Boolean dimension
  of a poset is bounded in terms of the tree-width of its cover graph,
  independent of its height. We show that the local dimension of a
  poset cannot be bounded in terms of the tree-width of its cover
  graph, independent of height. We also prove that the local dimension
  of a poset is bounded in terms of the path-width of its cover
  graph. In several of our results, Ramsey theoretic methods will be
  applied.
\end{abstract}

\maketitle

\section{Introduction}

We investigate combinatorial problems for finite posets. As has become
standard in the literature, we use the terms \textit{elements} and
\textit{points} interchangeably in referring to the members of the
ground set of a poset.  We write $x\parallel y$ in $P$ when $x$ and
$y$ are incomparable in a poset $P$, and we let $\Inc(P)$ denote the
set of all ordered pairs $(x,y)$ with $x\parallel y$ in $P$.  As a
binary relation, $\Inc(P)$ is symmetric.  Recall that a non-empty
family $\cgR$ of linear extensions of $P$ is called a
\textit{realizer} of $P$ when $x<y$ in $P$ if and only if $x<y$ in $L$
for each $L\in\cgR$.  Clearly, a non-empty family $\cgR$ of linear
extensions of $P$ is a realizer of $P$ if and only if for each
$(x,y)\in \Inc(P)$, there is some $L\in\cgR$ for which $x>y$ in $L$.
The \textit{dimension} of a poset $P$, denoted $\dim(P)$, as defined
by Dushnik and Miller in their seminal paper~\cite{bib:DusMil}, is the
least positive integer $d$ for which $P$ has a realizer $\cgR$ with
$|\cgR|=d$.

For an integer $n\ge2$, the \textit{standard example} $S_n$ is the\
height~$2$ poset with minimal elements $A=\{a_1,a_2,\dots,a_n\}$ and
maximal elements $B=\{b_1,b_2,\dots,b_n\}$. Furthermore, $a_i<b_j$ in
$S_n$ if and only if $i\neq j$. As noted in~\cite{bib:DusMil},
$\dim(S_n)=n$, for all $n\ge2$.  Also, dimension is clearly a
monotonic parameter, i.e., if $Q$ is a subposet of $P$, then
$\dim(Q)\le\dim(P)$.  Accordingly, a poset which contains a large
standard example as a subposet has large dimension.  On the other
hand, there are posets which do not contain the standard example $S_2$
as a subposet and nevertheless have large dimension.  This observation
is the poset analogue to the fact that there are triangle-free graphs
which have large chromatic number.

Quite recently, researchers have been investigating combinatorial
problems for two variations of the Dushnik-Miller concept for
dimension, known as \textit{Boolean dimension} and \textit{local
  dimension}.  The concept of Boolean dimension was introduced by
Gambosi, Ne\v{s}et\v{r}il and Talamo in a 1987 conference
paper~\cite{bib:GaNeTa-1}, with the full version~\cite{bib:GaNeTa-2}
appearing in journal form in 1990.  However, we use here the
definition of Boolean dimension which appears in a 1989 paper by
Ne\v{s}et\v{r}il and Pudl\'{a}k~\cite{bib:NesPud}.  This paper was
first presented in conference form in 1987.  Later, we will comment on
the distinction between the two definitions.
 
On the other hand, the quite new notion of local dimension is due to
Torsten Ueckerdt~\cite{bib:Ueck} who shared his ideas with
participants of the workshop on \textit{Order and Geometry} held in
Gu{\l}towy, Poland, September 14--17, 2016. Ueckerdt's new concept
resonated with researchers at the workshop, and it served to kindle
renewed interest in Boolean dimension as well.

Here is the definition for Boolean dimension.  For a positive integer
$d$, let $\bfII^d$ denote the set of all $0$--$1$ strings of
length~$d$. Such strings are also called \textit{bit strings}.  Let
$P$ be a poset and let $\cgB=\{L_1,L_2,\dots,L_d\}$ be a non-empty
family of linear orders on the ground set of $P$ (these linear orders
need not be linear extensions of $P$).  Also, let
$\tau:\bfII^{d}\rightarrow \{0,1\}$ be a function.  For each pair
$(x,y)$ of distinct elements of $P$, we form a bit string
$q(x,y,\cgB)$ of length~$d$ which has value $1$ in coordinate $i$ if
and only if $x<y$ in $L_i$.  The pair $(\cgB,\tau)$ is a
\textit{Boolean realizer}\footnote{In~\cite{bib:GaNeTa-2}, a pair
  $(\cgB,\tau)$ with $\cgB=\{L_1,L_2,\dots,L_d\}$, is considered a
  Boolean realizer only when there is some $i$ for which $L_i$ is a
  linear extension of $P$ and $\tau(x,y)=1$ implies $x<y$ in $L_i$.
  We prefer to drop both these restrictions, as is done
  in~\cite{bib:NesPud}.}  when for each pair $x,y$ of distinct
elements of $P$, $x<y$ in $P$ if and only if $\tau(q(x,y,\cgB))=1$.
The \textit{Boolean dimension of $P$}, denoted $\bdim(P)$, is the
least positive integer $d$ for which $P$ has a Boolean realizer
$(\cgB,\tau)$ with $|\cgB|=d$.  Clearly, $\bdim(P)\le\dim(P)$, since
if $\cgR=\{L_1,L_2,\dots, L_d\}$ is a realizer of $P$, we simply take
$\tau$ as the function which maps $(1,1,\dots,1)$ to $1$ while all
other bit strings of length~$d$ are mapped to~$0$.

Trivially, $\bdim(P)=1$ if and only if $P$ is either a chain or an
antichain\footnote{In~\cite{bib:FeMeMi}, Felsner, M\'{e}sz\'{a}ros and
  Micek consider pairs $x,y$ of not necessarily distinct elements of
  $P$ so a query $q(x,y,\cgB)$ has coordinate~$i$ set to~$1$ if and
  only if $x\le y$ in $L_i$.  With this restriction, the function
  $\tau$ is constrained to send the constant string $(1,1,\dots,1)$
  to~$1$, so that a non-trivial antichain has Boolean dimension~$2$.
  For all other posets, their definition and ours give exactly the
  same value for Boolean dimension.}.  Also, $\bdim(Q)\le \bdim(P)$
when $Q$ is a subposet of $P$. Clearly, $\bdim(P)=\bdim(P^*)$ where $P^*$
denotes the dual of $P$.  It is an easy exercise to show that if
$\bdim(P)=2$, then $\dim(P)=2$.  In~\cite{bib:GaNeTa-2}, Gambosi,
Ne\v{s}et\v{r}il and Talamo show that $\dim(P)=3$ if and only if
$\bdim(P)=3$.  However, their proof uses a more restrictive definition
of Boolean dimension.  In~\cite{bib:TroWal}, Trotter and Walczak
simplify the proof given in~\cite{bib:GaNeTa-2} and show that it
actually works for the more general notion of Boolean dimension we are
studying in this paper.  It is an easy exercise to show that all
standard examples have Boolean dimension at most~$4$.  In fact,
$\bdim(S_n)=n$ when $2\le n\le 4$ and $\bdim(S_n)=4$ when $n\ge4$.

Here is the definition for local dimension.  Let $P$ be a poset.  A
\textit{partial linear extension}, abbreviated $\ple$, of $P$ is a
linear extension of a subposet of $P$.  Whenever $\cgL$ is a family of
$\ple$'s of $P$ and $u\in P$, we set
$\mu(u,\cgL)=|\{L\in\cgL:u\in L\}|$.  In turn, we set
$\mu(\cgL)=\max\{\mu(u,\cgL):u\in P\}$.  A non-empty family $\cgL$ of
$\ple$'s of a poset $P$ is called a \textit{local realizer} of $P$ if
the following two conditions are satisfied:
\begin{enumerate}
\item If $x\le y$ in $P$, there is some $L\in\cgL$ for which $x\le y$
  in $L$;

\item if $(x,y)\in \Inc(P)$, there is some $L\in\cgL$ for which $x>y$
  in $L$.
\end{enumerate}
The \textit{local dimension of $P$}, denoted $\ldim(P)$, is defined as
\[\ldim(P)=\min\{\mu(\cgL):\cgL\text{ is a local realizer of }P\}.\]
Clearly, $\ldim(P)\le\dim(P)$ for all posets $P$, since any realizer
is also a local realizer. Also, $\ldim(P)=1$ if and only if $P$ is a
chain; $\ldim(Q)\le\ldim(P)$ if $Q$ is a subposet of $P$; and if $P^*$
is the dual of $P$, then $\ldim(P^*)=\ldim(P)$.  It is an easy
exercise to show that if $\ldim(P)=2$, then $\dim(P)=2$.  In
presenting his concept to conference participants,
Ueckerdt~\cite{bib:Ueck} noted that the local dimension of a standard
example is at most~$3$.  In fact, $\ldim(S_n)=n$ when $2\le n\le 3$
and $\ldim(S_n)=3$ when $n\ge3$.

In this paper, we give analogies and contrasts between
(Dushnik-Miller) dimension, Boolean dimension and local dimension.
Although our results touch on several other topics, we consider the
connections with structural graph theory, given in
Section~\ref{sec:sgt}, our main theorems.  A number of open problems
remain, and we give a summary listing in the closing section.

Our arguments will use the following notational conventions:
\begin{enumerate}
\item If $n$ is a positive integer, then we use the now standard
  notation $[n]$ to represent $\{1,2,\dots,n\}$.
\item Let $\cgL=\{L_1,L_2,\dots,L_t\}$ be a family of $\ple$'s of a
  poset $P$. If $x\in P$, and $\mu(x,\cgL)=m$, then there are integers
  $j_1<j_2<\dots<j_m$ so that $x$ is in $L_{j_\alpha}$ for each
  $\alpha\in[m]$. In this case, we will say that \textit{occurrence
    $\alpha$ of $x$ is in $L_{j_\alpha}$}.
\item We will make use of the general form of Ramsey's theorem: 
  For every triple $(k,h,r)$ of positive integers with $h\ge k$, there
  is a least positive integer $\Ram(k,h;r)$ so that if
  $n\ge \Ram(k,h;r)$ and $\varphi$ is any coloring of the $k$-element
  subsets of $[n]$ using $r$ colors, then there is an $h$-element
  subset $H$ of $[n]$ so that $\varphi$ maps all $k$-element subsets of
  $H$ to the same color.
\end{enumerate}

\section{Forcing Large Boolean Dimension and Large Local Dimension}

Since standard examples have small Boolean dimension and small local
dimension, it is of interest to explore what can cause these two
parameters to be large.  We start with an example of a well known
family of posets where dimension, local dimension and Boolean
dimension all grow together.

When $n\ge2$, we let $I_n$ denote the \textit{canonical interval
  order} whose elements are the closed intervals of the form $[i,j]$
where $i$ and $j$ are integers with $1\le i<j\le n$. The partial order
on $I_n$ is defined by setting $[i,j]<[k,l]$ in $I_n$ when $j<k$. As
is well known, the poset $I_n$ does not contain the standard example
$S_2$, but the dimension of $I_n$ goes to infinity with $n$. In fact,
the value of $\dim(I_n)$ is now known to within an additive constant
(see the remarks in~\cite{bib:BHPT}).  We now explain why both
$\ldim(I_n)$ and $\bdim(I_n)$ tend to infinity. We start with the
result for local dimension.

\begin{theorem}\label{thm:ldim-grow}
  For each $s\ge1$, if $n\ge \Ram(4,7;s^2)$, then $\ldim(I_n)>s$.
\end{theorem}
\begin{proof} 
  Suppose to the contrary that for some $s\ge1$, and
  $n\ge \Ram(4,7;s^2)$ we have $\ldim(I_n)\le s$. Let
  $\cgL=\{L_i:1\le i\le t\}$ be a local realizer for $I_n$ with
  $\mu(I_n,\cgL)\le s$. Consider a $4$-element subset $\{a,b,c,d\}$ of
  $[n]$ with $a<b<c<d$. Then there is some least positive integer
  $m\in[t]$ so that $[a,c]>[b,d]$ in $L_m$. We then set
  $\varphi(\{a,b,c,d\})=(\alpha,\beta)$ where occurrence $\alpha$ of
  $[a,c]$ is in $L_m$ and occurrence $\beta$ of $[b,d]$ is in
  $L_m$. Now, we have a coloring of the $4$-element subsets of $[n]$
  using $s^2$ colors.

  In view of our choice for the size of $n$, we know there is some
  $7$-element subset $H=\{a,b,c,d,e,f,g\}$ of $[n]$ and a color
  $(\alpha,\beta)$ so that all $4$-element subsets of $H$ are mapped
  to $(\alpha,\beta)$. We may assume, without loss of generality, that
  $a<b<c<d<e<f<g$. Now consider the subset $\{a,c,d,g\}$. Then let $m$
  be the least positive integer so that $[a,d]>[c,g]$ in $L_m$. Then
  occurrence $\alpha$ of $[a,d]$ is in $L_m$ as is occurrence $\beta$
  of $[c,g]$.

  Now consider the set $\{b,c,f,g\}$. Since occurrence $\beta$ of
  $[c,g]$ is in $L_m$, then the least $m'$ such that $[b,f] > [c, g]$ in $L_{m'}$ is $m'=m$ and occurrence $\alpha$ of $[b,f]$ is
  also in $L_m$.

  Now consider the set $\{b,e,f,g\}$. Since occurrence $\alpha$ of
  $[b,f]$ is in $L_m$, we know that occurrence $\beta$ of $[e,g]$ is
  also in $L_m$. Furthermore, we know that $[b,f]>[e,g]$ in $L_m$.
  Finally, we consider the set $\{a,b,d,f\}$ and conclude that
  $[a,d]>[b,f]$ in $L_m$. In particular $\alpha=\beta$. However, we
  have now shown that $[a,d]>[b,f]>[e,g]$ in $L_m$. This is a
  contradiction since $[a,d]<[e,g]$ in $I_n$.
\end{proof}

Here is the analogous result for Boolean dimension.

\begin{theorem}\label{thm:bdim-grow}
  For each $d\ge 1$, if $n\ge\Ram(4,6;2^d)$, then $\bdim(I_n)>d$.
\end{theorem}
\begin{proof}
  Suppose to the contrary that for some $d\ge1$, and
  $n\ge \Ram(4,6;2^d)$ we have $\bdim(I_n)\le d$.  Let $(\cgB,\tau)$
  be a Boolean realizer for $I_{n}$ with $\cgB=\{L_1,L_2,\dots,L_d\}$.
  Then for each $4$-element subset $\{a,b,c,d\}$ of $[n]$ with
  $a<b<c<d$, we define the coloring $\varphi$ by setting
  $\varphi(\{a,b,c,d\})= q([a,c],[b,d],\cgB)$.

  In view of our choice for $n$, we may assume that there is some
  binary string $\sigma$ of length $d$ and a $6$-element subset
  $H=\{a,b,c,d,e,f\}$ of $[n]$, such that $a<b<c<d<e<f$, so that
  $\varphi$ maps all $4$-element subsets of $H$ to $\sigma$.  In
  particular, $\varphi$ assigns the color $\sigma$ to the $4$-element
  subsets $\{a,b,c,e\}$ and $\{b,d,e,f\}$, that is,
  $q([a,c],[b,e],\cgB)=\sigma=q([b,e],[d,f],\cgB)$.

  Now let $i\in[d]$.  If $\sigma(i)=1$, then $[a,c]<[b,e]<[d,f]$ in
  $L_i$.  If $\sigma(i)=0$, then $[d,f]<[b,e]<[a,c]$ in $L_i$.
  However, this shows that
  $\sigma=q([a,c],[b,e],\cgB)=q([a,c],[d,f],\cgB)$. This is a
  contradiction since $[a,c]<[d,f]$ in $I_n$, so
  $\tau(q([a,c],[d,f],\cgB))=1$, but $[a,c]$ and $[b,e]$ are
  incomparable, so $\tau(q([a,c],[b,e],\cgB))=0$.
\end{proof}

Next, we present a family for which dimension and local dimension are
unbounded but Boolean dimension is bounded.  For a pair $(d,n)$ of
integers with $2\le d < n$, let $P(1,d;n)$ denote the poset consisting
of all $1$-element and $d$-element subsets of $[n]$ partially ordered
by inclusion.  We abbreviate the dimension, Boolean dimension and
local dimension of $P(1,d;n)$ as $\dim(1,d;n)$, $\bdim(1,d;n)$ and
$\ldim(1,d;n)$, respectively.  Dushnik~\cite{bib:Dush} calculated
$\dim(1,d;n)$ exactly when $d\ge 2\sqrt{n}$, and
Spencer~\cite{bib:Spen} showed that for fixed $d$,
$\dim(1,d;n)=\Theta(\log\log n)$.  Historically, there has been
considerable interest in the case where $d=2$.  Combining results of
Ho\c{s}ten and Morris~\cite{bib:HosMor} with estimates of Kleitman and
Markovsky~\cite{bib:KleMar}, the following theorem follows easily (see
the comments in~\cite{bib:BHPT}).

\begin{theorem}\label{thm:12n}
  For every $\epsilon>0$, there is an integer $n_0$ so that if $n>n_0$
  and
  \[
    s= \lg\lg n + 1/2 \lg\lg\lg n +1/2\lg{\pi} + 1/2,
  \]
  then $s-\epsilon < \dim(1,2;n)< s +1 +\epsilon$.
\end{theorem}

As a consequence, for almost all large values of $n$, we can compute
the value of $\dim(1,2;n)$ \textit{exactly}; for the remaining small
fraction of values, we are able to compute two consecutive integers
and say that $\dim(1,2;n)$ is one of the two.

We are not able to compute the value of $\ldim(1,2;n)$ as accurately,
but at least we can show that $\ldim(1,2;n)$ goes to infinity with
$n$.

\begin{theorem}\label{thm:ldim-grow-1}
  For each $s\ge1$, if $n\ge \Ram(3,4;s^2)$, then $\ldim(1,2;n)>s$.
\end{theorem}
\begin{proof}
  Fix $s\ge1$ and let $n\ge\Ram(3,4;s^2)$. We assume that 
  $\cgL=\{L_1,L_2,\dots,L_t\}$ is a local realizer for $P=P(1,2;n)$
  with $\mu(P,\cgL)\le s$ and argue to a contradiction. In the
  argument, we abbreviate the singleton sets in $P(1,2;n)$ by omitting
  braces, i.e., the singleton set $\{a\}$ will just be written as $a$.
  Now the partial order is that an integer $a\in[n]$ is less than a
  $2$-element set $S$ in $P(1,2;n)$ when $a\in S$.

  Now let $T=\{a,b,c\}$ be a $3$-element subset of $[n]$. We may
  assume without loss of generality that $a<b<c$. Since
  $b\not\in\{a,c\}$, there is some least integer $m\in[t]$ with
  $b>\{a,c\}$ in $L_m$. Then there is an ordered pair
  $(\alpha,\beta)\in[s]\times[s]$ of (not necessarily distinct)
  integers so that occurrence $\alpha$ of $b$ is in $L_m$ and
  occurrence $\beta$ of $\{a,c\}$ is in $L_m$. We then have a coloring
  $\varphi$ of the $3$-element subsets of $[n]$ using $s^2$
  colors. Since $n\ge\Ram(3,4;s^2)$, there is some color
  $(\alpha,\beta)$ and a $4$-element subset $H=\{a,b,c,d\}$ so that
  all $3$-element subsets of $H$ are assigned color
  $(\alpha,\beta)$. Again, we may assume without loss of generality
  that $a<b<c<d$.

  We consider first the $3$-element subset $\{a,b,d\}$ and note that
  there is some $m\in[t]$ for which $b>\{a,d\}$ in $L_m$.
  Furthermore, occurrence $\alpha$ of $b$ is in $L_m$ while occurrence
  $\beta$ of $\{a,d\}$ is in $L_m$. Now consider the subset
  $\{a,c,d\}$. Since occurrence $\beta$ of $\{a,d\}$ is in $L_m$, we
  must have occurrence $\alpha$ of $c$ in $L_m$ with $c>\{a,d\}$ in
  $L_m$.

  Now consider the subset $\{a,b,c\}$. Since occurrence $\alpha$ of
  $b$ is in $L_m$, we must then have $b>\{a,c\}$ in $L_m$. On the
  other hand, if we consider the subset $\{b,c,d\}$, since occurrence
  $\alpha$ of $c$ is in $L_m$, we must have $c>\{b,d\}$ in $L_m$. We
  then have $\{b,d\}<c<\{a,c\}<b$ in $L_m$, which is a contradiction
  to the fact that $b<\{b,d\}$ in every ple of $P(1,2;n)$ where $b$
  and $\{b,d\}$ appear.
\end{proof}

Since $P(1,2;n)$ is a subposet of $P(1,d;n+d-2)$, it follows that for
fixed $d\ge2$, $\ldim(1,d;n)$ tends to infinity with $n$.  However, as
we will soon see $\bdim(1,d;n)$ is bounded in terms of $d$.

For the family $P(1,d;n)$, every maximal element is comparable with
exactly $d$ elements.
A careful reading of the proof of Theorem 3.6 on page~259
in~\cite{bib:GaNeTa-2} shows that they have actually established the
following result.

\begin{theorem}\label{thm:bdim-Delta}
  Let $P$ be a poset of height~$2$.  If there is some positive integer
  $d$ so that each maximal element of $P$ is comparable with at most
  $d$ minimal elements, then $\bdim(P)\le 2d$.
\end{theorem}

The inequality in Theorem~\ref{thm:bdim-Delta} is obviously tight for
$d=1$.  We will now show that it is tight for $d\ge2$.  To accomplish,
we will show that $\bdim(1,d;n)=2d$, provided $n$ is sufficiently
large in terms of $d$.  The argument will make use of the following
elementary observation.  When $(\cgB,\tau)$ is a Boolean realizer of a
poset $P$, it is easy to see that a linear order $L_i$ in $\cgB$ can
be replaced\footnote{This statement does not apply for the definition
  of Boolean dimension used in~\cite{bib:FeMeMi}.}  with $L_i^*$, the
\textit{dual} of $L_{i}$, i.e., $x<y$ in $L_i$ if and only if $x>y$ in
$L_i^{*}$.  Of course, we must also make the obvious modification to
the map $\tau$.

\begin{theorem}\label{thm:bdim-delta-tight}
  For each $d\ge 2$, there is some positive integer $n_0$ so that if
  $n\ge n_0$, then $\bdim(1,d;n)=2d$.
\end{theorem}
\begin{proof}
  We already know that $\bdim(1,d;n)\le 2d$ for all $d\ge2$.  We fix a
  value of $d\ge2$, suppose that $\bdim(1,d;n)< 2d$ for each $n>d$
  and argue to a contradiction.

  Let $(\cgB,\tau)$ be a Boolean realizer for $P(1,d;n)$ with
  $\cgB=\{L_1,L_2,\dots,L_s\}$ such that $s<2d$.  As before, we
  take $\Min(P)=[n]$ with $\Max(P)$ the family of all
  $d$-element subsets of $[n]$.

  First, we apply Erd\H{o}s-Szekeres to the set $[n]$ of minimal
  elements of $P$ relative to the order of these elements in the
  linear orders in $\cgB$ to obtain a subset $A$ of $[n]$ that 
  appears either in increasing order or decreasing order for each
  $L_{i}\in\cgB$.  Using our previous remarks concerning duals of
  linear orders in $\cgB$, if $n$ is sufficiently large, we may assume
  there is a subset $A$ of $[n]$ with $|A|=2d+1$ so that the
  restriction of $L_j$ to $A$ is exactly the same as the restriction
  of $L_k$ to $A$ whenever $1\le j<k\le s$.  After relabeling, we may
  assume $A=\{1,2,\dots,2d+1\}$ so that $1<2<3<\dots<2d+1$ in $L_j$
  for each $j=1,2,\dots,s$.

  There are $2d$ ``gaps'' between consecutive elements of $A$ of the
  form $(i,i+1)$.  One of $i$ and $i+1$ is even and the other is odd.
  Now consider the maximal element $S=\{2,4,6,\dots,2d\}$. There 
  are $2d$ gaps and at most $2d-1$ linear orders in $\cgB$.  It follows 
  that there is some gap $(i,i+1)$ for which there is no integer $j$ with
  $ j\in [s]$ so that $i<S<i+1$ in $L_j$.  This implies that
  $q(i,S,\cgB)=q(i+1,S,\cgB)$ so that
  $\tau(q(i,S,\cgB))=\tau(q(i+1,S,\cgB))$.  This is a contradiction since
  one of $i$ and $i+1$ is in $S$ while the other is not.
\end{proof}

We comment in closing that Theorem~\ref{thm:bdim-Delta} can be easily
strengthened to yield the following result.

\begin{theorem}\label{thm:max-down}
  For every $d\ge1$, there is a constant $c_d$ so that if $P$ is poset
  and every maximal element of $P$ is comparable with at most $d$
  elements of $P$, then $\bdim(P)\le c_d$.
\end{theorem}

Furthermore, we note that Trotter and Walczak~\cite{bib:TroWal} proved
that if $P$ is a poset and $\ldim(P)\le 3$, then $\bdim(P)\le 8443$.
However, they also proved that for every $d\ge1$, there is a poset $P$
with $\bdim(P)\ge d$ and $\ldim(P)\le 4$.  Accordingly, in general,
neither Boolean dimension nor local dimension is bounded in terms of
the other.

\section{Basic Inequalities for Dimension}

Dimension, local dimension and Boolean dimension are all monotonic
parameters.  But, it is natural to ask whether they are
``continuous'', i.e., if $Q$ is a subposet of $P$ obtained by removing
a single point from $P$, are the values for $Q$ close to the
corresponding values for $P$?

For dimension, the following elementary result was proved by
Hiraguchi~\cite{bib:Hira}.  We include a short proof as the basic idea
will be important in the discussion to follow.

\begin{theorem}\label{thm:dim-1pt}
  Let $P$ be a poset on two or more points and let $x$ be an element
  of $P$. Then $\dim(P)\le 1+\dim(P-\{x\})$.
\end{theorem}
\begin{proof}
  Let $Q=P-\{x\}$, let $d=\dim(Q)$ and let $\{L_1,L_2,\dots,L_d\}$ be
  a realizer of $Q$. For an integer $i\in [d-1]$, let $M_i$ be any
  linear extension of $P$ such that the restriction of $M_i$ to $Q$ is
  $L_i$.  Let $Y$ be the ground set of $Q$ and let $D(x)$ consist of
  all points of $Q$ which are less than $x$ in $P$.  Dually, let
  $U(x)$ consist of all points of $Q$ which are greater than $x$ in
  $P$. Define $M_{d}$ and $M_{d+1}$ by:

  \begin{align*}
    M_d&=L_d(D(x))<x<L_d(Y-D(x))\quad\text{and}\\
    M_{d+1}&=L_d(Y-U(x))<x<L_d(U(x)).
  \end{align*}
  Clearly, $\{M_1,M_2,\dots,M_{d+1}\}$ is a realizer of $P$.
\end{proof}

We now prove the analogous inequality for local dimension, although
the argument is a bit more complex.

\begin{theorem}\label{thm:ldim-1pt}
  Let $P$ be a poset on two or more points and let $x$ be an element
  of $P$.  Then $\ldim(P)\le 1+\ldim(P-\{x\})$.
\end{theorem}
\begin{proof}
  Let $Q=P-\{x\}$.  We show that if $d=\ldim(Q)$, then
  $\ldim(P)\le d+1$.

  Now let $\cgL$ be a local realizer of $Q$. Clearly, we may assume
  that $\mu(y,\cgL)=d$ for every $y\in Q$. Let $y_0\in Q$ and relabel
  the $\ple$'s in $\cgL$ as $\{L_1,L_2,\dots,L_t\}$ so that
  $y_0\in L_i$ when $i\in[d]$. For each $i\in[d]$, let $Q_i$ be the
  subposet of $P$ determined by the ground set of $L_i$. It follows
  that if $u\in Q$, then $u\in Q_{i}$ for some $i\in[d]$.  Then for
  each $i\in[d]$, let $M_i$ be a linear extension of the subposet of
  $P$ determined by elements of $Q_{i}$ and $x$ for which the
  restriction of $M_i$ to $Q_i$ is $L_i$.

  Let $I(x)=\{u\in P:x\parallel u$ in $P\}$. If $I(x)=\emptyset$,
  then \[\{M_i:i\in[d]\}\cup\{L_j:d+1\le j\le t\}\] is a local
  realizer for $P$ and this would imply that $\ldim(P)=d$. So we may
  assume that $I(x)\neq\emptyset$.

  Let $W=\{w\in I(x): w\not\in Q_d, x>w$ in $M_i$ for all $i\in[d-1]$
  with $w\in Q_i\}$. Also, set $Z= \{z\in I(x): z\not\in Q_d, x<z$ in
  $M_i$ for all $i\in[d-1]$ with $z\in Q_i\}$. Note that
  $W\cap Z=\emptyset$.

  The ple $L_d$ has the block form $A<\{x\}<B$. Then let
  $A'= A\cap I(x)$ and $B'=B\cap I(x)$. We then define $\ple$'s $N_1$
  and $N_2$ as follows: The ground set of $N_1$ is
  $\{x\}\cup Q_d\cup W$ and the ground set of $N_2$ is
  $\{x\}\cup Q_d\cup Z$. These two $\ple$'s will have the following
  block form:
  \begin{align*}
    N_1&=A-A'<\{x\}<A'\cup B\cup W,\\
    N_2&=A\cup B'\cup Z<\{x\}< B-B'.
  \end{align*}
  Note that no element in $W$ is less than an element in $A-A'$, or
  else it would be comparable to $x$. The analogous assertion holds
  for elements in $Z$ and $B-B'$. Furthermore, the ordering of
  elements of $A-A'$ in $N_1$ is equal to the ordering of $A-A'$ in
  $L_{d}$. A similar assertion holds for elements of $A'\cup B$ in
  $N_{1}$, $B-B'$ in $N_{2}$, and $A\cup B'$ in $N_{2}$ when comparing
  to the ordering in $L_{d}$. It follows that:
  \[
    \cgL'=\{M_i:1\le i<d\}\cup\{L_j:d<j\le t\}\cup\{N_1,N_2\}
  \]
  is a local realizer for $P$ with $\mu(P,\cgL')=d+1$.
\end{proof}

We do not know whether the analogous result holds for Boolean
dimension.  In fact, here is the best inequality we have been able to
obtain concerning the removal of a single point.

\begin{theorem}\label{thm:bdim-1pt}
  Let $P$ be a poset on two or more points and let $x$ be an element
  of $P$.  Then $\bdim(P)\le 3+\ldim(P-\{x\})$.
\end{theorem}
\begin{proof}
  Let $(\cgB,\tau)$ be a Boolean realizer for $Q=P-\{x\}$, with
  $|\cgB|=\bdim(Q)=d$.  Label the linear orders in $\cgB$ as
  $\{L_1,L_2,\dots,L_d\}$.  For each $i\in[d]$, let $M_i$ be the
  linear order on the ground set of $P$ defined by setting
  $M_i=x<L_i$.  Next, we set $M_{d+1}=x<L_1^*$.

  Now let $L$ be any linear extension of $P$.  With a shift in
  subscripts and letting $Y$ be the ground set of $Q$, we follow the proof of Theorem~\ref{thm:dim-1pt} and
  set:
  \begin{align*}
    M_{d+2}&=L(D(x))<x<L(Y-D(x))\quad\text{and}\\
    M_{d+3}&=L(Y-U(x))<x<L(U(x)).
  \end{align*}
  Note that $M_{d+2}$ and $M_{d+3}$ are linear extensions of $P$.

  Then set $\hat{\cgB}=\{M_1,M_2,\dots, M_{d+3}\}$.  For a pair
  $(u,v)$ of distinct points of $P$, we claim that we can always
  determine whether $u$ is less than $v$ in $P$ based on the bits in
  the string $q(u,v,\hat{\cgB})$.  First, we consider the bits
  associated with the linear orders in
  $\{M_1,M_2,\dots,M_{d}, M_{d+1}\}$.  If one of $u$ and $v$ is $x$,
  these bits are constant; otherwise they are not.  Furthermore, if
  one of $u$ and $v$ is $x$, we can tell whether $u<v$ in $P$ from the
  bits associated with the linear extensions $M_{d+2}$ and $M_{d+3}$.
  If neither $u$ nor $v$ is $x$, then we can tell whether $u$ is less
  than $v$ in $P$ by applying $\tau$ to the bits associated with
  $\{M_1,M_2,\dots,M_d\}$.
\end{proof}

\subsection{Inequalities involving Width}

In his classic 1950 paper~\cite{bib:Dilw}, Dilworth observed in a
first page footnote that an immediate consequence of his chain
partitioning theorem is that the Dushnik-Miller dimension of a poset
is at most its width.  The standard examples show that this elementary
inequality is best possible.  To date, we have not been able to
determine the maximum local dimension of a poset of width $w$
($w\geq4$). While it is bounded above by $w$, we do not know if this
is a tight upper bound. The analogous question for Boolean dimension
also remains open.

Although it may seem surprising, we have been able to settle analogous
questions for more complex inequalities involving width.  As one such
example, the following inequality was proved by
Trotter~\cite{bib:Trot-2}.

\begin{theorem}\label{thm:P-max(P)}
  Let $P$ be a poset and let $A=\Max(P)$. If $P-A$ is non-empty and
  has width $w$, then $\dim(P)\le w+1$ and this is sharp.
\end{theorem}

In~\cite{bib:Trot-2}, a family $\{P_w:w\ge 2\}$ of posets is
constructed to show that the inequality in Theorem~\ref{thm:P-max(P)}
is tight for Dushnik-Miller dimension. These posets are shown in
Figure~\ref{fig:P-max(P)}.

\begin{figure}
\centering
\input{fig1.tex}
\caption{Showing the Inequality is Tight}
\label{fig:P-max(P)}
\end{figure}
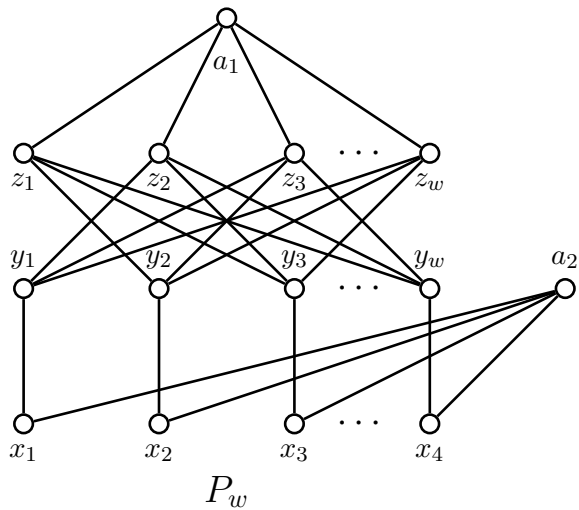

As noted in~\cite{bib:Trot-2}, it is straightforward to verify that
for each $w\ge3$, the poset $P_w$ is $(w+1)$-irreducible. However, it
is an easy exercise to show that all the posets in this family have
local dimension at most~$4$, and they have Boolean dimension at
most~$4$.  Therefore it remains to answer the following: Is the
inequality in Theorem~\ref{thm:P-max(P)} tight for local dimension or
for Boolean dimension? We will explain why the answer for both
parameters is yes, but we elect to postpone the argument until we have
discussed a second inequality involving width.

The following inequality was also proved in~\cite{bib:Trot-2}.

\begin{theorem}\label{thm:P-A}
  Let $A$ be an antichain in a poset $P$ with $P-A$ non-empty.  If the
  width of the subposet $P-A$ is $w$, then $\dim(P)\le 2w+1$.
\end{theorem}

The argument to show that this inequality is best possible is more
complex, and a construction to accomplish this task is given by
Trotter in a separate paper~\cite{bib:Trot-1}.  We now analyze a
``one-sided'' variation of that construction.

For a pair $(n,w)$ of positive integers, we define a poset $P=P(n,w)$
containing $nw+n^{w}$ points. The subposet $P-\Max(P)$ contains $nw$
elements that form a disjoint sum of $w$ chains each of size $n$:
$C_{1}+C_{2}+\ldots+C_{w}$. We label the points of $C_{i}$ as
$x_{i,1}<x_{i,2}<\ldots<x_{i,n}$. For each sequence
$\sigma= (j_1,j_2,\dots,j_w)\in [n]^{w}$ of positive integers taken
from $[n]$, there is a maximal element $a_\sigma$ of $P$ with
$a_\sigma$ covering $x_{i,j_i}$ in $P$ for each $i\in[w]$.  Note that
there are $n^w$ maximal elements in $P$, and in the argument below, we
will denote the set $\Max(P)$ of maximal elements of $P$ just as $A$.

We also require a special case of a result which has become known as
the ``Product Ramsey Theorem,'' given in the classic
text~\cite{bib:GrRoSp} as Theorem~5 on page~113. However, we will use
slightly different notation in discussing this result.

When $A_1,A_2,\dots,A_t$ are $k$-element subsets of
$B_1,B_2,\dots,B_t$, respectively, we refer to the Cartesian product
$A_1\times A_2\times\dots\times A_t$ as a \textit{$\bfk^t$-grid} in
$B_1\times B_2\times\dots\times B_t$. Here is a formal statement of
the version of the Product Ramsey Theorem we will use.

\begin{theorem}\label{thm:prod-ramsey}
  Let $(k,t,h,r)$ be a $4$-tuple of positive integers with $h\geq
  k$. There exists a least positive integer $n_0=\PRam(k,t,h,r)$ such
  that if $n\ge n_0$, $g$ is an $\bfn^{t}$-grid and $\varphi$ is a
  coloring of all $\bfk^{t}$-grids in $g$ with $r$ colors, then there
  exists an $\bfh^{t}$-grid $g'$ in $g$ such that all $\bfk^{t}$-grids
  in $g'$ are mapped to the same color by $\varphi$.
\end{theorem}

With these preparatory remarks in hand, here is the result we will
prove.  This theorem shows that the inequality in
Theorem~\ref{thm:P-max(P)} is best possible for both local dimension
and Boolean dimension.

\begin{theorem}\label{thm:local:P-max(P)}
  For every $w\ge1$, there is an integer $n_0$ so that if $n\ge n_0$,
  then $\ldim(P(n,w))=\bdim(P(n,w))=w+1$. Note that $w$ is the width
  of $P(n,w)-\Max(P(n,w)).$
\end{theorem}
\begin{proof}
  We give full details of the proof for local dimension, which is
  slightly more complicated.  At the end, we will outline how an
  argument for Boolean dimension can be structured.

  Since $\ldim(P(n,w)\le\dim(P(n,w))\le w+1$, we need only show
  that $\ldim(P(n,w))\ge w+1$, provided $n$ is sufficiently large.
  This assertion holds trivially when $w=1$, so we will fix a
  value $w\ge2$, assume that $\ldim(P(n,w))\le w$ for all $n$ and
  argue to a contradiction.

  We consider a large, but unspecified value of $n$, and we let
  $\cgL=\{L_i:i\in[t]\}$ be a local realizer of $P(n,w)$ with
  $\mu(P,\cgL)\le w$. Clearly, we may assume $\mu(z,\cgL)=w$ for every
  $z\in P(n,w)$.

  Next, we describe a coloring $\varphi$ of the $\bfII^w$ grids in
  $C_1\times C_2\times\dots\times C_w$. For each $i\in[w]$, consider a
  $2$-element subset $S_i=\{j_i,j'_i\}$ of $C_i$ with $j_i<j'_i$. Note
  that $g=S_1\times S_2\times\dots\times S_w$ is a $\bfII^{w}$
  grid. With the grid $g$, we associate the antichain
  $\{x_{i,j'_i}:i\in[w]\}$ and an element $a(g)$ of $\Max(P)$. We set
  $a(g)=a_{\sigma}$, where $\sigma=(j_{1},j_{2},\ldots,j_{w})$.
  Clearly, $a(g)$ is incomparable with each element of the antichain.
  Therefore for each $i\in[w]$, there is a least positive integer
  $m_i\in[t]$ so that $a(g)<x_{i,j'_i}$ in $L_{m_i}$.  Then we set
  $\varphi(g)= ((\alpha_1,\beta_1),(\alpha_2,\beta_2),
  \dots,(\alpha_w,\beta_w))$, where occurrence $\alpha_i$ of $a(g)$
  and occurrence $\beta_i$ of $x_{i,j'_i}$ is in $L_{m_i}$, for each
  $i\in[w]$.

  The number of colors used by $\varphi$ is $w^{2w}$, thus we take
  $n\geq \PRam(5,w,2,w^{2w})$. Theorem~\ref{thm:prod-ramsey} implies
  that there exists a $\bfV^{w}$-grid
  $H_{1}\times H_{2}\times \ldots\times H_{w}$ such that every
  $\bfII^{w}$-grid within it is assigned the same color:
  \[((\alpha_1,\beta_1),(\alpha_2,\beta_2),\dots,
    (\alpha_w,\beta_w)).\] We relabel the elements of $P$ so that
  $H_{i}=\{x_{i,1}, x_{i,2},x_{i,3},x_{i,4},x_{i,5}\}$ with
  $x_{i,1}< x_{i,2}<x_{i,3}<x_{i,4}<x_{i,5}$, for each $i\in[w]$.

  Consider the $\bfII^{w}$-grids of the form
  $g_{a}=S_1\times S_2\times\dots\times S_w$ where
  $S_i=\{x_{i,1},x_{i,a}\}$ with $a\in\{2,3,4,5\}$, for each
  $i\in[w]$. These grids show that there is a sequence
  $(m_1,m_2,\dots,m_w)$ of not necessarily distinct integers so that
  for each $i\in[w]$, occurrence $\beta_i$ of $x_{i,2}$, $x_{i,3}$,
  $x_{i,4}$ and $x_{i,5}$ all occur in $L_{m_i}$.

  Let us show that the elements of the sequence $(m_1,m_2,\dots,m_w)$
  are pairwise distinct. Suppose, for a contradiction, that
  $m_{1}=m_{2}$, noting that this argument can be applied for the case
  where any other two elements of the sequence are equal. Let
  $S_{i}=\{x_{i,4},x_{i,5}\}$ for $i\in[w]$ and consider the
  $\bfII^{w}$-grids
  $g_{1}=\{x_{1,1},x_{1,2}\}\times S_{2}\times S_{3}\times\ldots\times
  S_{w}$, and
  $g_{2}=S_{1}\times\{x_{2,1},x_{2,2}\}\times S_{3}\times\ldots\times
  S_{w}.$ We must have $a(g_{1})<x_{1,4}$ in $L_{m_{1}}$, and
  $a(g_{2})<x_{2,2}$ in $L_{m_{2}}$. As $m_{1}=m_{2}$ this implies
  $a(g_{1}),x_{1,4},a(g_{2}),$ and $x_{2,2}$ appear in $L_{m_{1}}$.
  Since $x_{1,4}<a(g_{2})$ in $P$, it follows that
  $a(g_{1})<x_{1,4}<a(g_{2})<x_{2,2}$ in $L_{m_{1}}$. This is not
  possible as $x_{2,2}<a(g_{1})$ in $P$. Therefore the integers in the
  sequence $(m_1,m_2,\dots,m_w)$ are all distinct.
  
  Now let $\sigma=(2,2,2,\dots,2)$ and $\sigma'=(3,3,3,\dots,3)$. It
  follows that $x_{i,2}<a_\sigma<x_{i,3}<a_{\sigma'}<x_{i,4}$ in
  $L_{m_i}$ for each $i\in[w]$. This accounts for all $w$ of the
  occurrences of $a_\sigma$ and $a_{\sigma'}$. As a consequence, there
  is no $\ple$ $L$ in $\cgL$ with $a_\sigma>a_{\sigma'}$ in $L$. Since
  $a_{\sigma}$ is incomparable to $a_{\sigma'}$ in $P$ this implies
  that $\cgL$ is not a local realizer for $P$. The contradiction
  completes the proof of the theorem for local dimension.

  Here is an outline of the argument for Boolean dimension.  As
  before, suppose that $P=P(n,w)$ has a Boolean realizer $(\cgB,\tau)$
  with $|\cgB|=w$ and argue to a contradiction when $n$ is
  sufficiently large. First, use the Product Ramsey theorem to assume
  that, after relabeling of the chains in $P$ and the linear orders in
  $\cgB$, we have the following two properties:
\begin{itemize}
\item[(1)] for each $(i,j)\in[w]\times[w]$, the elements of $C_i$
  appear as a block in $L_j$, and

\item[(2)] for each $a\in\Max(P)$, and for each $i\in[w]$, if $a$
  covers a point $x\in C_i$ and $x<x'$ in $C_i$, then $a$ is between
  $x$ and $x'$ in $L_i$.
\end{itemize}
Once this structure has been identified, it is easy to see that for
\textit{every} bit-string $\sigma$ of length~$w$, there is some pair
$(a,a')$ of distinct maximal elements such that $q(a,a',\cgB)=\sigma$.
Clearly, this results in a contradiction if we simply choose $\sigma$
such that $\tau(\sigma)=1$.
\end{proof}

The original construction given in~\cite{bib:Trot-1} has an antichain
$A$ with $n$ chains $C_1+C_2+\dots+C_n$ below $A$ and $n$ chains
$D_1+D_2+\dots+D_n$ above $A$.  Now the size of $A$ is $n^{2w}$, where
each element in $A$ covers exactly one element from each $C_{i}$ and
is covered by exactly one element from each $D_{j}$, $i,j\in[n]$.
Using this construction, it is straightforward to modify the argument
given above to show that the inequality in Theorem~\ref{thm:P-A} is
best possible for both local dimension and Boolean dimension.
 
\subsection{Dimension and Size}

The following well known inequality is due to
Hiraguchi~\cite{bib:Hira}.

\begin{theorem}\label{thm:dim-size}
  If $n\ge2$ and $|P|\le 2n+1$, then $\dim(P)\le n$.
\end{theorem}

The family of standard examples shows that the preceding theorem is
best possible for Dushnik-Miller dimension.  Accordingly, it is of
interest to determine (or at least estimate) the maximum value of the
Boolean dimension and the maximum value of the local dimension of a
poset on $n$ points.

Resolving this question for Boolean dimension is the principal result
in Ne\v{s}et\v{r}il and Pudlak's 1989 paper~\cite{bib:NesPud}.

\begin{theorem}\label{thm:bdim-size}
  The maximum value of the Boolean dimension of a poset on $n$ points
  is $\Theta(\log n)$.
\end{theorem}

The lower bound for the preceding theorem results from a simple
counting argument.  Consider an integer $n=2m$ and the posets on $2m$
points with $\{a_1,a_2,\dots,a_m\}\subseteq \Min(P)$ and
$\{b_1,b_2,\dots,b_m\}\subseteq\Max(P)$.  Clearly, there are $2^{m^2}$
such posets.  If they all have Boolean dimension at most $d$, then we
must have
\[
  (2m!)^{d}2^{2^d}\ge 2^{m^2}.
\]

This implies that $d=\Omega(\log n)$.  The argument given
in~\cite{bib:NesPud} to show that the maximum Boolean dimension of a
poset on $n$ points is $O(\log n)$ is more complex.

Quite recently, Kim, Martin, Masa\v{r}\'ik, Shull, Smith, Uzzell and
Wang~\cite{bib:KMMSSUW} have settled the analogous question for local
dimension using clever probabilistic methods.  Both upper and lower
bounds of their proof are non-trivial.

\begin{theorem}\label{thm:ldim-size}
  The maximum value of the local dimension of a poset on $n$ points is
  $\Theta(n/\log n)$.
\end{theorem}

\subsection{Dimension and the Complement of Antichains}

The following inequality was proved independently by
Trotter~\cite{bib:Trot-2} and Kimble~\cite{bib:Kimb}.

\begin{theorem}\label{thm:|P-A|}
  Let $A$ be an antichain in a poset $P$ and let $n=|P-A|$.  Then
  $\dim(P)\le \max\{2,n\}$.
\end{theorem}

The standard examples again show that the inequality in
Theorem~\ref{thm:|P-A|} is best possible.  Moreover, this inequality
coupled with the fact that $\dim(P)$ is at most the width of $P$
yields a simple proof of Hiraguchi's inequality.

For local dimension we have the following analogue, a result where
Theorem~\ref{thm:ldim-size} plays an important role.

\begin{theorem}\label{thm:ldim-|P-A|}
  The maximum value of the local dimension of a poset $P$ consisting
  of an antichain $A$ and $n$ other points is $\Theta(n/\log n)$.
\end{theorem}
\begin{proof}
  The argument for the lower bound in Theorem~\ref{thm:ldim-size}
  results from considering height~$2$ posets with $n$ minimal elements
  and $n$ maximal elements and showing that among them, there is (at
  least) one whose local dimension is $\Omega(n/\log n)$.
  Accordingly, the same lower bound applies in this theorem as well.

  The upper bound is a bit more complicated \footnote{This part of the
    proof is a result of conversations in 2016 with S.~Felsner,
    P.~Micek and V.~Wiechert.}, and we find it convenient to prove a
  slightly stronger result, i.e., we show that the local dimension of
  a poset $P$ is $O(n/\log n)$ when the ground set of $P$ can be
  partitioned as $A\cup X\cup Y$ where
  \begin{itemize}
  \item[(1)] $A$ is a maximal antichain in $P$;

  \item[(2)] each point of $X$ is less than some point in $A$;

  \item[(3)] each point of $Y$ is greater than some point in $A$; and
  
  \item[(4)]$|X|=|Y|=n$.
  \end{itemize}
  We now build a local realizer
  $\cgL=\cgL_1\cup\cgL_2\cup\cgL_3\cup \cgL_4$ of $P$.  We start by
  setting $\cgL_1=\{L_1,L_2\}$ where $X<A<Y$ in $L_1$, $X<A<Y$ in
  $L_2$ and the restriction of $L_1$ to $A$ is the dual of the
  restriction of $L_2$ to $A$.  Using Theorem~\ref{thm:ldim-size}, we
  take the family $\cgL_2$ to be a local realizer of the subposet $Q$
  determined by $X\cup Y$ with $\mu(u,\cgL_2)$ being $O(n/\log n)$ for
  each point $u\in Q$.

  Next, we construct a family $\cgL_3$ of $\ple$'s of $X\cup A$ so
  that
  \begin{itemize}
  \item[(1)] for each incomparable pair $(x,a)$ with $x\in X$ and
    $a\in A$, there is some $L\in\cgL_3$ with $x>a$ in $L$; and
  \item[(2)]$\mu(u,\cgL_3)$ is $O(n/\log n)$ for each $u\in X\cup A$.
  \end{itemize}
  We begin by taking an arbitrary partition of $X$ as
  $X=X_1\cup X_2\cup\dots\cup X_s$ where each subposet $X_i$ has size
  $m=n/s$. As usual in arguments of this type, we are assuming $s$ and
  $m$ are integers.  For each $i\in[s]$, we let $\cgU_i$ denote the
  family of all upsets of $X_i$.  Considering $\cgU_i$ as partially
  ordered by inclusion, it is clear that $\cgU_i$ can be partitioned
  into at most $\binom{m}{\lceil m/2\rceil}$ chains, as $\cgU_{i}$ is
  a subposet of the Boolean lattice (or subset lattice).

  Now let
  $S_1\subsetneq S_2\subsetneq S_{3}\subsetneq\dots\subsetneq S_r$ be
  any chain in this partition of $\cgU_i$.  We build a $\ple$ $L$
  using the following recursion.  Set $D_1=S_1$ and let
  $D_i=S_i-S_{i-1}$ for $2\le i\le r$.  An element $x\in X$ will be in
  $L$ if and only if $x\in S_r$.  Second, we have $x>y$ in $L$ if
  there are integers $i$ and $j$ with $1\le i<j\le r$ so that
  $x\in D_i$ and $y\in D_j$. The order $L$ assigns to a pair
  $x,y\in S_r$ when there is some $i$ for which $x,y\in D_i$ is
  arbitrary.

  To complete the definition of $L$ we add those elements $a\in A$
  such that there is some $i$ for which $a$ is incomparable with all
  elements of $S_i$ and comparable with all elements of $X_i-S_i$.  Of
  course, we place $a$ immediately under the lowest element of $S_i$
  in $L$.

  Now we count frequencies.  Each element of $X$ is in a unique
  subposet $X_i$.  So, being generous
  $\mu(x,\cgL_3) \le\binom{m}{\lceil m/2\rceil}$.  On the other hand,
  for each $i\in [s]$, an element $a\in A$ appears in at most one
  $\ple$ associated with chains in the partition of $\cgU_i$.  It
  follows that $\mu(a,\cgL_3)\le s$.

  So to optimize the construction, we choose $s$ so that
  $s=\binom{m}{m/2}$.  This yields that $\mu(u,\cgL_3)$ is
  $O(n/\log n)$ for every $u\in X\cup A$.

  To complete the proof, the preceding construction is then repeated
  in a symmetric manner to obtain a family $\cgL_4$ for $Y\cup A$.
\end{proof}

For Boolean dimension, we have been able to show that there is a
constant $C$ such that $\bdim(P)\le \lceil 2n/3\rceil+C$ when $P$
contains an antichain $A$ and $n$ other points.  We do not include the
details as we feel the result is most likely far from best possible.

\subsection{Dimension and the Product of Chains}

For positive integers $k$ and $d$, let $\mathbf{k}^{d}$ denote the
Cartesian product of $d$ copies of a $k$-element chain. As is well
known, for all $k\ge2$, $\dim(\mathbf{k}^d)=d$.  It is an easy
application of the Product Ramsey Theorem to show that for each
$d\ge1$, there is an integer $k_d$ so that if $k\ge k_d$, then
$\bdim(\mathbf{k}^d)= \ldim(\mathbf{k}^d)=d$.  However, we are
completely unable to settle whether or not $k_d=2$ when $d\ge2$.  An
easy counting argument shows that
$\bdim(\mathbf{2}^d)=\Omega(d/\log d)$, but it might be the case that
$\bdim(\mathbf{2}^d)=d$.  We know even less about the situation with
local dimension.

\subsection{Components and Blocks}

We assume that the reader is familiar with basic concepts of graph
theory, including the following terms: connected and disconnected
graphs, components, cut vertices and $k$-connected graphs for an
integer $k\ge2$. Recall that when $G$ is a graph, a connected induced
subgraph $H$ of $G$ is called a \emph{block} of $G$ when $H$ is a
maximal subgraph with no cut
vertex.

Here are the analogous concepts for posets.  A poset $P$ is said to be
\emph{connected} if its cover graph is connected.  A subposet $B$ of
$P$ is said to be \emph{convex} if $y\in B$ whenever $x,z\in B$ and
$x<y<z$ in $P$.  Note that when $B$ is a convex subposet of $P$, the
cover graph of $B$ is an induced subgraph of the cover graph of $P$.
A convex subposet $B$ of $P$ is called a \emph{component} of $P$ when
the cover graph of $B$ is a component of the cover graph of $P$.  A
convex subposet $B$ of $P$ is called a \emph{block} of $P$, when the
cover graph of $B$ is a block in the cover graph of $P$.

As is well known, when $P$ is a disconnected poset with components
$C_1,C_2,\dots,C_t$, for some $t\ge2$,
$\dim(P)=\max(\{2\}\cup\{\dim(C_i):1\le i\le t\})$.  Readers may note
that the preceding observation is just a special case of the formula
for the dimension of a \emph{lexicographic sum} (see page~23
in~\cite{bib:Trot-Book}).  For local dimension, it is an easy exercise
to show that $\ldim(P)\le 2+\max\{\ldim(C_i):1\le i\le t\}$, but we do
not know whether this inequality is best possible.

The corresponding result for Boolean dimension is more complicated and
is due to M\'{e}sz\'{a}ros, Micek and Trotter~\cite{bib:MeMiTr}.

\begin{theorem}\label{thm:bdim-components}
  Let $P$ be a disconnected poset with components $C_1,C_2,\dots,C_t$,
  for some $t\ge2$.  If $d=\max\{\bdim(C_i): 1\le i\le t\}$, then
  $\bdim(P)\le2 + d+ 4\cdot 2^d$.
\end{theorem}

The inequality in Theorem~\ref{thm:bdim-components} cannot be improved
dramatically, since it is shown in~\cite{bib:MeMiTr} that for large
$d$, there is a disconnected poset $P$ with $\bdim(P)=\Omega(2^d/d)$
and $\bdim(C)\le d$ for every component $C$ of $P$.

The situation with blocks is more complex, even for Dushnik-Miller
dimension.  In~\cite{bib:TrWaWa}, Trotter, Walczak and Wang prove the
following result for Dushnik-Miller dimension.

\begin{theorem}\label{thm:dim-blocks}
  If $d\ge1$ and $\dim(B)\le d$ for every block of a poset $P$, then
  $\dim(P)\le d+2$.  Furthermore, this inequality is best possible.
\end{theorem}

Neither the proof of the inequality in Theorem~\ref{thm:dim-blocks},
nor the proof that the inequality is best possible is elementary.
Surprisingly, however, there is no analogous result for local
dimension, as Bosek, Grytczuk and Trotter~\cite{bib:BoGrTr} prove that
for every $d\ge4$, there is a poset $P$ with $\ldim(P)\ge d$, such
that $\ldim(B)\le 3$ whenever $B$ is a block in $P$.

However, on the issue of blocks, Boolean dimension behaves like
Dushnik-Miller dimension, as the following inequality is proved
in~\cite{bib:MeMiTr}.

\begin{theorem}\label{thm:bdim-blocks}
  If $d\ge1$ and $\bdim(B)\le d$ for every block $B$ of a poset $P$,
  then $\bdim(P)\le 9+d+18\cdot 2^d$.
\end{theorem}

Again, this inequality cannot be improved dramatically, as it is shown
in~\cite{bib:MeMiTr} that for large $d$, there is a poset $P$ with
$\bdim(P)=\Omega(2^d/d)$ and $\bdim(B)\le d$ for every block $B$ of
$P$.

\section{Planar Posets and Dimension}

A poset $P$ is \textit{planar} if its order diagram can be drawn in
the plane without edge crossings.  If a poset is planar, then its
cover graph is planar, although the converse does not hold in general.
It is easy to see that the standard example $S_n$ is planar when
$2\le n\le 4$ and non-planar when $n\ge5$.

In Figure~\ref{fig:Kelly}, we show a construction due to
Kelly~\cite{bib:Kell} showing that for all $n\ge5$, the non-planar
poset $S_n$ is a subposet of a planar poset.  This specific figure is
a diagram where $n=5$, but it should be clear how we intend that the
diagram should be modified for other values of $n$. Of course, the
Kelly posets show that there are planar posets with arbitrarily large
dimension.

\begin{figure}
\input{kelly.tex}
\caption{The Kelly Construction}
\label{fig:Kelly}
\end{figure}
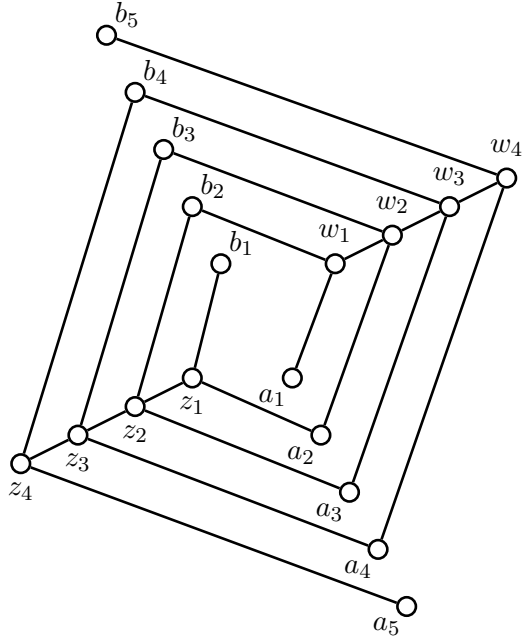

In retrospect, the Kelly posets should have prompted research on the
following questions:

\begin{enumerate}
\item Must a planar poset with large dimension have large height?
\item Must a planar poset with large dimension have many minimal
  elements (and many maximal elements)?
\item Must a planar poset with large dimension contain two large
  chains with all points in one incomparable with all points in the
  other?
\item Must a planar poset with large dimension contain a large
  standard example?
\end{enumerate}

However, these natural questions lay dormant for more than $20$~years,
so here is a compact summary of work done in the last five years.  The
first three questions in this listing have been answered in the
affirmative.  However, the last question in the list has been open for
nearly $30$ years.

In 2014, Streib and Trotter~\cite{bib:StrTro} proved that for every
positive integer $h$, there is a least positive integer $c_h$ so that
if $P$ is a poset of height~$h$ and the cover graph of $P$ is planar,
then $\dim(P)\le c_h$.  The proof given in~\cite{bib:StrTro} merely
established the existence of $c_h$ and gave no useful information
about its size.  However, an exponential upper bound was given
in~\cite{bib:JMOW}, and more recently, two groups have announced a
polynomial upper bound on $c_h$.  Joret, Micek, Ossona de Mendez and
Wiechert have shown how their results in~\cite{bib:JMOW} can be
extended to obtain this conclusion.  Meanwhile, Kozik, Krawczyk, Micek
and Trotter~\cite{bib:KKMT} have a much more complicated argument which
yields a better exponent.  From Below, Joret, Micek and
Wiechert~\cite{bib:JoMiWi} showed that the $c_h\ge 2h-2$.

For planar posets, Joret, Micek and Wiechert~\cite{bib:JoMiWi} have a
linear upper bound, i.e., they show that a planar poset of height $h$
has dimension at most $192h+96$.  They have also given $4h/3-2$ as a
lower bound.

In~\cite{bib:TroWan}, Trotter and Wang proved that the dimension of a
planar poset with $t$ minimal elements is at most $2t+1$. They also
showed that this inequality is tight for $t=1$ and $t=2$. For $t\ge3$,
they were only able to show that there is a planar poset with $t$
minimal elements which has dimension $t+3$.  Using duality, analogous
statements hold for maximal elements.  Note, however, that there are
no statements of this type for posets with planar cover graphs, since
as pointed out in~\cite{bib:StrTro}, for every $d\ge1$, there is a
poset $P$ with a zero and a one such that $\dim(P)\ge d$ and the cover
graph of $P$ is planar.

In~\cite{bib:HSTWW}, Howard, Streib, Trotter, Walczak and Wang proved
that for each $k\ge 1$, there is a constant $d_k$ so that if $P$ is a
poset which does not contain two chains $C_1$ and $C_2$ each of size
$k$ such that all points of $C_1$ are incomparable with all points of
$C_2$, then the dimension of $P$ is at most $d_k$.

In~\cite{bib:FeTrWi}, Felsner, Trotter and Wiechert showed that if $P$
is a poset and the cover graph of $P$ is outerplanar, then
$\dim(P)\le 4$.  They also gave an example to show that the inequality
is best possible. This same example shows that the inequality is tight
for Boolean dimension and local dimension.  The argument for Boolean
dimension is trivial, while the argument for local dimension has the
same spirit as the proof of Theorem~\ref{thm:ldim-grow}.  We leave the
details of this proof as an exercise.

In~\cite{bib:NesPud}, Ne\v{s}et\v{r}il and Pudl\'{a}k note that the
Kelly posets have Boolean dimension at most~$4$, and they asked
whether Boolean dimension is bounded for the class of planar posets.
This challenging question remains open.  We note that it is
conceivable (although we consider it very unlikely) that Boolean
dimension is bounded for planar posets but unbounded for posets with
planar cover graphs.

In presenting his concept of local dimension to conference
participants, Ueckerdt noted that standard examples have local
dimension at most~$3$, and it is easy to see that in fact, the Kelly
posets have local dimension at most~$3$.  This leads naturally to the
question: Do planar posets have bounded local dimension?  However,
this question has recently been answered in the negative by Bosek,
Grytczuk and Trotter~\cite{bib:BoGrTr}.

\section{Connections with Structural Graph Theory}\label{sec:sgt}

In this section, we explore which variants of dimension can be bounded
in terms of path-width or tree width. For the sake of completeness, we
include here the basic definitions of tree-width and path-width. Let
$G$ be a graph with vertex set $V(G)$. A \textit{tree-decomposition}
of $G$ is a pair $(T,\cgB)$ where $T$ is a tree with vertex set
$V(T)$, and $\cgB=\{B_t:t\in V(T)\}$ is a family of subsets of $V(G)$
satisfying:
\begin{enumerate}
\item[($T_1$)] for each $v\in V(G)$ there exists $t\in V(T)$ with
  $v\in B_t$; and for every edge $uv$ in $G$ there exists $t\in V(T)$
  with $u,v\in B_t$;
\item[($T_2$)] for each $v\in V(G)$, if $v\in B_t \cap B_{t''}$ for
  some $t,t''\in V(T)$, and $t'$ lies on the path in $T$ between $t$
  and $t''$, then $v\in B_{t'}$.
\end{enumerate}
It is common to refer to the tree $T$ as the \emph{host tree} in the
tree-decomposition, and when $t\in V(T)$, the induced subgraph
$G[B_t]$ of $G$ is referred to as a \emph{bag}. Note that, $|B_t|$ is
just the number of vertices of $G[B_{t}]$.

The \emph{width} of a tree-decomposition $(T,\cgB)$ is defined as
\[\max_{t\in V(T)}\{|B_{t}|-1\}.\]  The \emph{tree-width} of
$G$, $\tw(G)$, is the minimum width of a tree-decomposition of $G$.

A tree-decomposition $(T,\{B_t:t\in V(T)\})$ is called a
\emph{path-decomposition} when the host tree $T$ is a path.  In turn,
the \textit{path-width} of $G$, $\pw(G)$, is the minimum width of a
path-decomposition of $G$. Observe that $\pw(G)\geq \tw(G)$ since
every path-decomposition of $G$ is a tree-decomposition of $G$.

We encourage readers to consult the discussion of connections between
Dushnik-Miller dimension and structural graph theory as detailed
in~\cite{bib:JMMTWW}, \cite{bib:TroWan} and~\cite{bib:TroWal}. Here we
provide a quick summary of highlights.

The first major result linking dimension and structural graph theory
is due to Joret, Micek, Milans, Trotter, Walczak and
Wang~\cite{bib:JMMTWW}, who proved that the dimension of a poset is
bounded as a function of its height and the tree-width of its cover
graph. More formally, they showed that for each pair $(t,h)$ of
positive integers, there is a least positive integer $d(t,h)$ so that
if $P$ is a poset of height $h$ and the tree-width of the cover graph
of $P$ is~$t$, then $\dim(P)\le d(t,h)$. A poset of height~$1$ is an
antichain and has dimension at most~$2$, so it is of interest to study
$d(t,h)$ only when $h\ge2$.  Trotter and Moore~\cite{bib:TroMoo}
showed that $d(1,h)=3$ for all $h\ge2$, and Joret, Micek, Trotter,
Wang, and Wiechert~\cite{bib:JMTWW} showed that $d(2,h)\le 1276$ for
all $h\ge2$.  As is well known, Kelly posets have cover graphs with
path-width at most~$3$, so $d(t,h)$ goes to infinity with $h$ when
$t\ge3$.

Joret, Micek and Wiechert~\cite{bib:JoMiWi} have recently shown that
for fixed $t\ge3$, $d(t,h)$ grows exponentially with $h$.  The best
bound to date in the general case is due to Joret, Micek, Ossona de
Mendez and Wiechert~\cite{bib:JMOW}, where they prove:

\begin{equation}
2^{\Omega(h^{\lfloor (t-1)/2 \rfloor})}
 \leq d(t,h)
 \leq 4^{\binom{t+3h-3}{t}}.
\end{equation}

Now we turn our attention to analogous results for Boolean dimension
and local dimension.  In 2016, Micek and Walczak~\cite{bib:MicWal}
proved that the Boolean dimension of a poset is bounded in terms of
the path-width of its cover graph, independent of its height.  In
2017, Felsner, M\'{e}sz\'{a}ros and Micek~\cite{bib:FeMeMi} proved
that in fact, the Boolean dimension of a poset is bounded in terms of
the tree-width of its cover graph, independent of its height.

Now on to local dimension.  We will first prove the following result
which asserts that the local dimension of a poset is bounded in terms
of the path-width of its cover graph, independent of its height.

\begin{theorem}\label{thm:ldim-pw}
  For every $t\ge1$, there is a least positive integer $d(t)$ so that
  if $P$ is a poset whose cover graph has path-width~$t$, then
  $\ldim(P)\le d(t)$.
\end{theorem}

The details of the proof show that $d(t)$ is
$O(5^{(t+1)^2})$. However, we will then show that the local dimension
of a poset is \textit{not} bounded in terms of the tree-width of its
cover graph independent of its height.

\begin{theorem}\label{thm:ldim-tw}
  For every $d\ge1$, there exists a poset $P$ with $\ldim(P)>d$ such
  that the cover graph of $P$ has tree-width at most~$3$.
\end{theorem}

\subsection{Local Dimension and Path-Width}

Here we give the proof of Theorem~\ref{thm:ldim-pw}.  Our argument
requires some preliminary material on a concept introduced by
Kimble~\cite{bib:Kimb-pc}.  The \textit{split of a poset $P$} is the
height~$2$ poset $Q$ whose minimal elements are $\{x':x\in P\}$ and
whose maximal elements are $\{x'':x\in P\}$. Furthermore, for all
$x,y\in P$ not necessarily distinct, $x'<y''$ in $Q$ if and only if
$x\le y$ in $P$.

The following well known result is an easy exercise, but it is stated
here for emphasis.

\begin{theorem}\label{thm:split}
  Let $Q$ be the split of a poset $P$. Then
  $\dim(P)\le \dim(Q)\le 1+\dim(P)$.
\end{theorem}

Recent work in dimension theory has made use of a variant of the
notion of a split. Let $P$ be a poset and let $X$ denote the ground
set of $P$. The \textit{split-in-place} of $P$ is the poset $R$
obtained as follows:
\begin{enumerate}
\item The ground set of $R$ is disjoint union of three sets
  $X\cup X'\cup X''$.
\item $X'=\{x':x\in X\}=\Min(R)$ and $X''=\{x'':x\in X\}=\Max(R)$.
\item The subposet of $R$ determined by $X$ is $P$.
\item In $R$, for each $x\in X$, $x'$ is only covered by $x$, and
  $x''$ only covers $x$.
\end{enumerate}
Observe that the split of $P$ is a subposet of the split-in-place of
$P$.

Essentially the same argument used to prove Theorem~\ref{thm:split}
yields the following extension.

\begin{theorem}\label{thm:split-in-place}
  Let $Q$ be the split and let $R$ be the split-in-place of a poset
  $P$.  Then $\dim(P)\le \dim(Q)\le \dim(R)\le 1+\dim(P)$.
\end{theorem}

We note that there is no analogue of this theorem for Boolean dimension. Indeed, while the Boolean dimension of the canonical interval order is unbounded, it is easy to show that the split of any interval order has Boolean dimension at most 6. Here is the analogue of the preceding
theorem for local dimension.

\begin{lemma}\label{lem:ldim-split}
  Let $Q$ be the split and let $R$ be the split-in-place of a poset
  $P$.  Then $\ldim(P)\le\ldim(R)\le 2\ldim(Q)-1$ and
  $\ldim(Q)\le \ldim(R)\le 2+\ldim(P)$.
\end{lemma}
\begin{proof}
  The inequalities $\ldim(P)\le\ldim(R)$ and $\ldim(Q)\le\ldim(R)$
  hold since both $P$ and $Q$ are subposets of $R$.

  Setting $s=\ldim(Q)$, we show that $\ldim(R)\le 2s-1$.  Let
  $\cgL=\{L_1,L_2,\dots,L_t\}$ be a local realizer of $Q$ with
  $\mu(u,\cgL)= s$ for every $u\in Q$. Recall that the ground set of
  $R$ is $X\cup X'\cup X''$, where $X$ is the ground set of $P$. For
  each $i\in[t]$, let $X'_i$ consist of those elements $x'\in X'$
  which are in $L_i$ and let $X''_i$ consist of those elements
  $x''\in X''$ which are in $L_i$. Then let $X_i$ consist of those
  elements $x\in X$ for which either $x'\in X'_i$ or $x''\in X''_i$.

  For the poset $R$, let $M_i$ be a $\ple$ whose ground set is
  $X_i\cup X'_i\cup X''_i$ such that the restriction of $M_i$ to
  $X'_i\cup X''_i$ is $L_i$. Checking the necessary details it can be
  seen that $\cgM=\{M_1,M_2,\dots,M_t\}$ is a local realizer of $R$,
  with $\mu(R,\cgM)\le 2s$. However, since for each $x\in P$ there is
  some $i$ with $x'<x''$ in $L_{i}$ it follows that
  $\mu(R,\cgM)\leq 2s-1$.

  Next, we show that $\ldim(R)\le 2+\ldim(P)$.  Let $d=\ldim(P)$ and
  let $\cgL=\{L_1,L_2,\dots,L_t\}$ be a local realizer of $P$ with
  $\mu(x,\cgL)= d$ for all $x\in P$.  For each $i\in[t]$, let $X_i$ be
  the ground set of the $\ple$ $L_i$. We then modify $L_i$ as follows:
  for each $x\in X_i$, we add $x'$ immediately under $x$ and we add
  $x''$ immediately over $x$.

  It remains to witness the incomparabilities $(a,b)$ where both $a$
  and $b$ are in $X'$ or both are in in $X''$. Construct two linear
  extensions of $R$ as follows. Let $L$ be a linear extension of
  $P$. Take $M_{0}$ to be the linear extension of $R$ with block
  structure $X'<X<X''$, where the restriction to each of $X,X'$ and
  $X''$ is ordered according to the corresponding elements in
  $L$. Similarly, define $M_{0}'$ to be the linear extension of $R$
  with block structure $\overline{X}'<X<\overline{X}'',$ where
  $\overline{X}'$ and $\overline{X}''$ are ordered dually to $L$, and
  $X$ is ordered according to $L$. Now, we can see that
  $\cgM=\{M_0,M'_0\}\cup\{M_i:i\in[t]\}$ is a local realizer for $R$
  with $\mu(R,\cgM) \le 2+d$.
\end{proof}

For the remainder of the proof, we fix a positive integer $t$ and let
$P$ be a poset whose cover graph has path-width at most~$t$.  We will
then show that $\ldim(P)$ is $O(5^{(t+1)^2})$.  Let $Q$ be the split
of $P$, and let $R$ be the split-in-place of $P$.  The basic idea for
the remainder of the argument is to prove that
the path-width of
the cover graph of $R$ is $t+1$. We will use this to show that the local dimension of
$Q$ is bounded in terms of $t$.  The conclusion of our theorem will
then follow from Lemma~\ref{lem:ldim-split}.

We let $\bbP_n$ denote the path whose vertex set is $[n]$ with
vertices $i$ and $j$ from $[n]$ adjacent in $\bbP_n$ if and only if
$|i-j|=1$.

Then let $G$ be the cover graph of $P$ and let $H$ be the cover graph
of $R$. 

Let $\cgP_{G}=(\bbP_{m},\{B_{t}:t\in[m]\})$ be a path-decomposition of
$G$ of width $t$. If $u\in P$ we define
$a_{\cgP_{G}}(u)=\min\{t\in [m]:u\in B_{t}\}$ and
$b_{\cgP_{G}}(u)=\max\{t\in [m]:u\in B_{t}\}$. Since $\cgP_{G}$ is a
path-decomposition, it follows that $u\in B_{t}$ if and only if
$a_{\cgP_{G}}(u)\leq t\leq b_{\cgP_{G}}(u)$. Thus we define the
\emph{interval} of $u$ in $\cgP_{G}$ to be the set of consecutive
integers
$\Int_{\cgP_{G}}(u)=\{t\in [m]: a_{\cgP_{G}}(u)\leq t \leq
b_{\cgP_{G}}(u)\}$. We may assume that the endpoints of the intervals
in $\cgP_{G}$ are distinct and every bag contains an endpoint, i.e.,
for each $i\in[m]$, there is exactly one vertex $u\in P$ with
$i\in\{a_{\cgP_{G}}(u),b_{\cgP_{G}}(u)\}$.

Let us fix the path-decomposition for $H$ to be
$\cgP=(\bbP_{3m},\{B'_{t}:t\in[3m]\})$, where
\[B'_{3i-j}=
  \begin{cases}
    B_{i}\cup\{u''\}\quad&\text{if }j=0\text{ and }i=a_{\cgP_{G}}(u)\text{ for some }
    u\in P,\\
    B_{i}\cup\{u'\}\quad&\text{if }j=1\text{ and }i=a_{\cgP_{G}}(u)\text{ for some }
    u\in P,\\
    B_{i}\quad&\text{if }j=2\text{ and }i=a_{\cgP_{G}}(u)\text{ for some }
    u\in P,\\
    
    B_{i}\quad&\text{if }j\in\{0,1,2\}\text{ and }i\neq
    a_{\cgP_{G}(u)} \text{ for all } u\in P.
  \end{cases}
\]
Since the path-decomposition of $H$ is now fixed we let $n=3m,$ and we
adopt similar notation for intervals in $\cgP$ as done above. That is,
if $u\in R$, then $a_{u}=\min\{t\in [n]:u\in B'_{t}\},$
$ b_{u}=\max\{t\in [n]:u\in B'_{t}\},$ and
$\Int(u)=\{t\in[n]:a_{u}\leq t\leq b_{u}\}$. Note that $\cgP$ has
width $t+1$ and that it satisfies the following properties:
\begin{enumerate}
\item The endpoints of the intervals in $\cgP$ are distinct, i.e., for
  each $i\in[n]$, there is at most one vertex $u\in R$ with
  $i\in\{a_{u},b_{u}\}$.
\item For every $u\in P$, $b_{u'}<a_{u''}$.
\item For every pair $(u,v)$ of (not necessarily distinct) elements of
  $P$, if $\Int(v)$ intersects either of $\Int(u')$ and $\Int(u'')$,
  then it contains both of them.
\end{enumerate}

We now begin to use properties of $\cgP$ to build a local realizer
$\cgL$ for $Q$.  To avoid a proliferation of primes and double primes
in the presentation, we will adopt the following conventions: the
letter $x$, sometimes written with subscripts, will always denote a
minimal element of $Q$.  Dually, the letter $y$ will always denote a
maximal element of $Q$. Also, we take $X$ as the set of all minimal
elements of $Q$ while $Y$ is the set of all maximal elements of
$Q$. We let $\Inc(X,Y)$ denote the set of all pairs
$(x,y)\in X\times Y$ with $(x,y)\in\Inc(Q)$.
 
We begin by including two linear extensions $L_0$ and $L'_0$ in $\cgL$
such that for all $(x,y)\in X\times Y$, $x<y$ in both $L_0$ and
$L'_0$, in which:
\begin{enumerate}
\item the restriction of $L_0$ to $X$ is the dual of the restriction
  of $L'_0$ to $X$, and
\item the restriction of $L_0$ to $Y$ is the dual of the restriction
  of $L'_0$ to $Y$.
\end{enumerate}
Given a set of ple's $\cgL'$ which satisfy the following condition,
then $\cgL=\cgL'\cup\{L_{0},L_{0}'\}$ is a local realizer for $Q$.

\medskip
\noindent\textbf{Reversing Min-Max Pairs.}
For each pair $(x,y)\in\Inc(X,Y)$, there is some $L\in \cgL'$ with
$x>y$ in $L$.

\medskip Of course, we must take care to keep $\mu(z,\cgL)$ bounded in
terms of $t$ for all $z\in Q$. We begin by taking $\varphi$ as a
proper coloring of the graph $G$ in the sense that for each pair of
distinct vertices $u,v\in P$ we have $\varphi(u)\neq\varphi(v)$ when
$\Int(u)\cap\Int(v)\neq\emptyset$. Let us see why such a coloring exists
using $t+1$ colors. Let $\cgP'$ be the path-decomposition of $G$
resulting from restricting the bags of $\cgP$ to only contain elements
of $G$. Consider the ordering of $V(G)$ given by the left endpoints of
intervals in $\cgP'$. We may greedily color vertices of $G$ according
to this ordering. Since $\cgP'$ witnesses that
$\pw(G)\leq t$, the resulting coloring does not use more
than $t+1$ colors. Without loss of generality, we may assume that $\varphi$
uses the integers in $[t+1]$ as colors.

Second, for each $z\in X\cup Y$, we let $\st(z)$ denote the set of all
points $u\in P$ such that $\Int(z)\subseteq \Int(u)$. Note that
$|\st(z)|\le t+1$. We then define a coloring $\pi$ of the elements of
$X\cup Y$.  The colors used by $\pi$ are vectors of length $t+1$ and
the coordinates are taken from $\{0,1,2,3\}$, so $\pi$ uses $4^{t+1}$
colors. Below $\pi(z)(i)$ is the $i$-th coordinate in the
$(t+1)$-tuple of $\pi(z)$. Note that there is at most one
$u\in\st(z)$ with $\varphi(u)=i$, because $\varphi$ is a proper
coloring (by definition of $\varphi$). For each $z\in X\cup Y$ and
each $i\in[t+1]$, we set:
\begin{enumerate}  \setcounter{enumi}{-1}
\item $\pi(z)(i)=0$ if there is no element $u\in\st(z)$ with
  $\varphi(u)=i$.
\item $\pi(z)(i)=1$ if there is an element $u\in\st(z)$ with
  $\varphi(u)=i$ and $z<u$ in $R$.
\item $\pi(z)(i)=2$ if there is an element $u\in\st(z)$ with
  $\varphi(u)=i$ and $z>u$ in $R$.
\item $\pi(z)(i)=3$ if there is an element $u\in\st(z)$ with
  $\varphi(u)=i$ and $z\parallel u$ in $R$.
\end{enumerate}

Next, we define a coloring $\tau$ of ordered pairs of elements from
$X\cup Y$. The colors used by $\tau$ are $(t+1)\times(t+1)$ matrices
with all entries taken from $\{0,1,2,3,4\}$, so $\tau$ uses
$5^{(t+1)^2}$ colors.  For each pair $(z,w)$ of elements of $X\cup Y$
and each pair $(i,j)\in [t+1]^{2}$, we set:
\begin{enumerate}  \setcounter{enumi}{-1}
\item $\tau(z,w)(i,j)=0$ if there is no pair $(u,w)$ with
  $u\in\st(z)$, $v\in\st(w)$, $\varphi(u)=i$ and $\varphi(v)=j$.
\item $\tau(z,w)(i,j)=1$ if there is a pair $(u,w)$ with $u\in\st(z)$,
  $v\in\st(w)$, $\varphi(u)=i$, $\varphi(v)=j$ and $u<v$ in $P$.
\item $\tau(z,w)(i,j)=2$ if there is a pair $(u,w)$ with $u\in\st(z)$,
  $v\in\st(w)$, $\varphi(u)=i$, $\varphi(v)=j$ and $u>v$ in $P$.
\item $\tau(z,w)(i,j)=3$ if there is a pair $(u,w)$ with $u\in\st(z)$,
  $v\in\st(w)$, $\varphi(u)=i$, $\varphi(v)=j$ and $u\parallel v$ in
  $P$.
\item $\tau(z,w)(i,j)=4$ if there is a pair $(u,w)$ with $u\in\st(z)$,
  $v\in\st(w)$, $\varphi(u)=i$, $\varphi(v)=j$ and $u=v$ in $P$.   
\end{enumerate}

Let $x,y\in R$. We say \emph{$x$ is left of $y$} if and only if
$b_{x}<a_{y}$. Under the same conditions, we say \emph{$y$ is right of
  $x$}. Now define a coloring $\sigma$ of the pairs in $\Inc(X,Y)$
using $4$-tuples of the form $(\alpha_1,\alpha_2,\alpha_3,\alpha_4)$.
The first coordinate $\alpha_1$ is $0$ if $x$ is left of $y$ and $1$
if $x$ is right of $y$. The remaining three coordinates are defined by
setting $\alpha_2=\pi(x)$, $\alpha_3=\pi(y)$ and
$\alpha_4=\tau(x,y)$. Clearly, $\sigma$ uses
$2\cdot4^{2(t+1)}\cdot 5^{(t+1)^2}$ colors.

Since the number of colors used by $\sigma$ is bounded in terms of
$t$, to complete the proof, it suffices to show that for each color
$\Gamma=(\alpha_1,\alpha_2,\alpha_3,\alpha_4)$ used by $\sigma$, we
can determine a family $\cgL(\Gamma)$ of $\ple$'s so that (1)~For each
$(x,y)\in\Inc(X,Y)$ with $\sigma(x,y)=\Gamma$, there is some
$L\in\cgL(\Gamma)$ with $x>y$ in $L$; and (2)~For each $z\in X\cup Y$,
$\mu(z,\cgL(\Gamma))$ is bounded in terms of $t$.

Fix a color $\Gamma=(\alpha_1,\alpha_2,\alpha_3, \alpha_4)$ and
consider the subset $S(\Gamma)$ of $\Inc(X,Y)$ consisting of all pairs
$(x,y)\in\Inc(X,Y)$ with $\sigma(x,y)=\Gamma$.  We will assume that
$\alpha_1=0$, i.e., if $(x,y)\in S(\Gamma)$ then all our pairs will
have $x$ left of $y$. From the details of the argument, it will be
clear that the case $\alpha_1=1$ is symmetric. Of course, we will also
assume that the set $S(\Gamma)$ is non-empty.

The next part of the proof will involve four claims. We begin by
proving the following.

\noindent
\textbf{Claim 1.}  Let $x_1,x_2\in X$ and $y_1,y_2\in Y$. Then the
following two statements hold:

\begin{enumerate}
\item If $(x_1,y_1)$ and $(x_2,y_1)$ are in $S(\Gamma)$, and $y_1$ is
  left of $y_2$, then $\tau(x_1,y_2)=\tau(x_2,y_2)$.  In particular,
  $(x_1,y_2)\in S(\Gamma)$ if and only if $(x_2,y_2)\in S(\Gamma)$.
\item If $(x_2,y_1)$ and $(x_2,y_2)$ are in $S(\Gamma)$, and $x_1$ is
  left of $x_2$, then $\tau(x_1,y_1)=\tau(x_1,y_2)$.  In particular,
  $(x_1,y_1)\in S(\Gamma)$ if and only if $(x_1,y_2)\in S(\Gamma)$.
\end{enumerate}
\begin{proof}
  Let us prove~(1), noting that the proof for~(2) follows from a
  similar argument. Suppose $\tau(x_{1},y_{2})(i,j)=k$ for some
  $(i,j)\in[t+1]^{2}$ and $k\in\{0,1,2,3,4\}$. We show that
  $\tau(x_{2},y_{2})(i,j)=k$. This results in the following five cases

  \begin{itemize}
  \item Assume that $k=0$. If $\pi(y_{2})(j)=0$, then
    $\tau(x_{2},y_{2})=0$. However, if $\pi(y_{2})(j)\neq 0$ then
    $\pi(x_{1})(i)=0$. Since $\sigma(x_{1},y_{1})=\sigma(x_{2},y_{1})$
    this implies in particular that $\pi(x_{1})(i)=\pi(x_{2})(i)$. So
    $\pi(x_{2})(i)=0$ and therefore $\tau(x_{2},y_{2})=0$.
    
  \item If $k=4$, then $i=j$ and there exists
    $u\in \st(x_{1})\cap\st(y_{2})$. Since $x_{1}$ is left of $y_{1}$
    and $y_{1}$ is left of $y_{2}$, then $u\in \st(y_{1})$. Therefore
    $\tau(x_{1},y_{1})(i,i)=4$. This implies
    $\tau(x_{2},y_{1})(i,i)=4$, as
    $\sigma(x_{1},y_{1})=\sigma(x_{2},y_{1})$. Thus $u\in\st(x_{2})$
    and it follows that $\tau(x_{2},y_{2})(i,i)=4$.

  \item For the case where $k=1$, there is $u\in\st(x_{1})$ and
    $v\in\st(y_{2})$ such that $\varphi(u)=i$, $\varphi(v)=j,$ and
    $u<v$ in $P$. Let $u=u_{1}u_{2}\ldots u_{m}=v$ be a path in $G$
    that witnesses the comparability $u<v$. It follows that
    $u_{l}\in\st(y_{1})$ for some $l\in[m]$. Suppose
    $\varphi(u_{l})=j'$. Let us assume that
    $\tau(x_{1},y_{1})(i,j')=1$, noting that the case where
    $\tau(x_{1},y_{1})(i,j')=4$ follows from an analogous
    argument. Since $\tau(x_{1},y_{1})(i,j')=1$, this is witnessed by
    $u\in \st(x_{1})$ and $u_{l}\in \st(y_{1})$. Since
    $\sigma(x_{1},y_{1})=\sigma(x_{2},y_{1})$ and
    $\tau(x_{1},y_{1})(i,j')=1$, we conclude that
    $\tau(x_{2},y_{1})(i,j')=1$. Therefore there is $u'\in\st(x_{2})$
    with $\varphi(u')=i$ and $u'<u_{l}$. We now have $u'<v$ and
    therefore $\tau(x_{2},y_{2})(i,j)=1$.
  \item The case where $k=2$ follows from an argument analogous to the
    one for $k=1$.
  \item We have shown that $\tau(x_{1},y_{2})\neq 3$ if and only if
    $\tau(x_{2},y_{2})\neq 3.$ Thus the result holds when $k=3$.
  \end{itemize}

\end{proof}

\noindent
\textbf{Claim 2.}  Let $S\subseteq S(\Gamma)$. Then the following two
statements hold:

\begin{enumerate}
\item If there is some $z\in X$ such that $(z,y)\in S(\Gamma)$
  whenever $(x,y)\in S$, then the set $S$ is reversible.
\item If there is some $w\in Y$ such that $(x,w)\in S(\Gamma)$
  whenever $(x,y)\in S$, then the set $S$ is reversible.
\end{enumerate}
\begin{proof}
  We prove the first statement and note that the proof for~(2) is
  symmetric. We argue by contradiction and assume there is some
  $k\ge2$ for which there is a strict alternating cycle\
  $\mathcal{S}=\{(x_i,y_i):i\in [k]\}$ contained in $S$.
  Without loss of generality, we may assume that the pairs of this
  cycle have been labeled so that $y_1$ is left of $y_i$ for all
  $i\in\{2,3,\dots,k\}$. Note that $(x_{i},y_{i})\in S$ and
  $x_{i}\parallel y_{i}$ in $R$ for all $i\in[k]$. Since
  $(z,y_{i})\in S(\Gamma),$ it follows that $z\parallel y_{i}$ for all
  $i\in[k]$. Because $(x_1,y_1)$ and $(z,y_1)$ are in $S(\Gamma)$ with $y_1$ left of $y_2$, Claim~1 guarantees $(x_{1},y_{2})\in S(\Gamma)$ since $(z,y_2)\in S(\Gamma)$,
  thus $x_{1}\parallel y_{2}$.  This is not possible
  since $x_{1}<y_{2}$ in $\mathcal{S}$.
\end{proof}

We consider the pairs in $S(\Gamma)$ as edges in a bipartite graph
$G(\Gamma)$ whose vertex set is $X\cup Y$ with vertex $x\in X$
adjacent to vertex $y\in Y$ in $G(\Gamma)$ when $(x,y)\in S(\Gamma)$.
In general, the graph $G(\Gamma)$ may be disconnected and some of the
components may just be isolated vertices. Regardless, since a vertex
from $X\cup Y$ belongs to at most one component of $G(\Gamma)$, it is
enough to consider a subset of $S(\Gamma)$ consisting of pairs
determining a non-trivial component of $G(\Gamma)$. Let $C$ be such a
component, let $S_C$ be the edge set of $C$, and let $X_C$ and $Y_C$
be, respectively, the subsets of $X$ and $Y$ which are incident with
at least one edge in $S_C$. Also, let $x_0$ be the left-most element
of $X_C$.

Then using the graph-theoretic concept of distance in a connected
graph, for each edge $(x,y)\in S_C$, we define
$\rho(x,y)=\min\{\dist(x,x_{0}),\dist(y,x_{0})\}$ to be the distance
from the edge $(x,y)$ to the vertex $x_0$. For each non-negative
integer $s$, we let $S_C(s)$ denote the set of edges $(x,y)\in S_C$
with $\rho(x,y)=s$. Note that $S_C(0)$ is just the set of edges
$(x_0,y)\in S_C$ where $y\in Y_C$. The set $X_C(s)$ consists of all
vertices $x\in X_C$ incident with an edge in $S_C(s)$. The set
$Y_C(s)$ is defined analogously. It is obvious that for each
$x\in X_C$, there are at most two values of $s$ for which
$x\in X_C(s)$. Furthermore, if there are two values, then they are
consecutive integers and the smaller of the two is odd.  Similarly, if
$y\in Y_C$, there are at most two values of $s$ for which
$y\in Y_C(s)$. If there are two values, they are consecutive integers
and the smaller of the two is even.

\medskip
\noindent
\textbf{Claim 3.}  The following two statements hold:
\begin{enumerate}
\item If $s$ is a non-negative even integer, $(x_1,y)\in S_C(s)$ and
  $(x_2,y)\in S_C(s+1)$, then $x_1$ is left of $x_2$.
\item If $s$ is an odd positive integer, $(x,y_1)\in S_C(s)$ and
  $(x,y_2)\in S_C(s+1)$, then $y_1$ is right of $y_2$.
\end{enumerate}
\begin{proof}
  First, suppose that $s=0$. Then since $(x_1,y)\in S_C(0)$, we know
  that $x_{1}=x_0$ and therefore $x_{1}$ is left of $x_{2}$ since
  $x_{0}$ was chosen to be the left-most element of $X_{C}$. Now, we
  argue by contradiction. Let $s$ be the least positive integer for
  which one of the two statements of the claim fails.

  If $s$ is a positive even integer and the claim fails for the pairs
  $(x_1,y)\in S_{C}(s)$ and $(x_2,y)\in S_{C}(s+1)$, then $x_2$ is
  left of $x_1$.  Let $(x_1,y_1)$ be any edge in $S_C$ so that
  $\rho(x_1,y_1)=s-1$.  Since the claim holds for $s-1$, we know that
  $y_1$ is right of $y$.  By Claim~1, we conclude that
  $(x_2,y_1)\in S_C$, so $\rho(x_2,y_1)\le s-1$ because $y_{1}$ is at
  distance $s-1$ from $x_{0}$. Therefore $\rho(x_{2},y)\leq s$, which
  contradicts the fact that $\rho(x_{2},y)\in S_{C}(s+1)$. Therefore
  $s$ is not a positive even integer.

  A similar contradiction is reached when $s$ is a positive odd
  integer, and with this observation, the proof of the claim is
  complete.
\end{proof}

Accordingly, to complete the proof of our theorem, it is enough to
show that for each non-negative integer $s$ and component $C$, there
is a family $\cgL_C(s)$ of ple's with ground set
$X_{C}(S)\cup Y_{C}(S)$ so that
\begin{enumerate}
\item for each $(x,y)\in S_C(s)$, there is some $L\in\cgL_C(s)$ with
  $x>y$ in $L$; and
\item $\mu(x,\cgL_C(s))$ and $\mu(y,\cgL_C(s))$ are bounded in terms
  of $t$, for every $x\in X_C(s)$ and every $y\in Y_C(s)$.
\end{enumerate}

The case $s=0$ is easy since all the edges in $S_C(0)$ are of the form
$(x_0,y)$, and clearly the set of such pairs is reversible. Similarly,
the case $s=1$ is handled by Claim~2, since it asserts that the set
$S_C(1)$ is reversible.

Now we fix an integer $s\ge2$. Suppose first that $s$ is even.  For
each $x\in X_C(s)$, there is a unique right-most point $w\in Y_C$ with
$(x,w)\in S_C(s-1)$.  We call $w$ the \textit{right-parent} of $x$.
For each $w\in Y_C$, we then let $X_C(s,w)$ denote those elements
$x\in X_C(s)$ for which $w$ is the right-parent of $x$. Clearly, when
$w_1\neq w_2$, the sets $X_C(s,w_1)$ and $X_C(s,w_2)$ are disjoint.

For a vertex $w\in Y_C$ for which $X_C(s,w)\neq\emptyset$, we then let
$Y_C(s,w)$ denote the set of all $y\in Y_C(s)$ for which there is some
$x\in X_C(s,w)$ with $(x,y)\in S_C(s)$.

\medskip
\noindent
\textbf{Claim 4.} If $w_1$ and $w_2$ are distinct elements of $Y_C$,
then $Y_C(s,w_1)\cap Y_C(s,w_2)=\emptyset$.
\begin{proof}
  Suppose to the contrary that there is some
  $y\in Y_C(s,w_1)\cap Y_C(s,w_2)$.  Choose elements
  $x_1\in X_C(s,w_1)$ and $x_2\in X_C(s,w_2)$ such that
  $(x_1,y),(x_2,y)\in S_C(s)$.  Without loss of generality, we may
  assume that $w_1$ is left of $w_2$. 
  Since $(x_2,w_2)\in S_C(s-1)$ and $(x_2,y)\in S_C(s)$ where $s$ is even, Claim 3 guarantees $y$ is left of $w_2$.
  Then by Claim~1, $(x_1,w_2)\in S_C$. Clearly, $\rho(x_1,w_2)$ is
  either $s-1$ or $s$. However, if $\rho(x_1,w_2)=s$, then the pairs
  $(x_1,w_1)$ and $(x_1,w_2)$ violate Claim~3. Also, if
  $\rho(x_1,w_2)= s-1$, $w_1$ is not the right parent of $x_1$.  The
  contradiction completes the proof of the claim.
\end{proof}

For each $w\in Y_C$ for which $X_C(s,w)\cup Y_C(s,w)$ is non-empty, we
form a $\ple$ $L_C(s,w)$ whose ground set is $X_C(s,w)\cup Y_C(s,w)$.
In view of Claim~2, we may assume that $x>y$ in $L_C(s,w)$ for every
pair $(x,y)\in S_C(s)$ with $x\in X_C(s,w)$. In view of Claim~4, for
each $z\in X_C\cup Y_C$, there is at most one element $w\in Y_C$ for
which $z$ is in the ground set of $L_{C}(s,w)$.

The proof when $s$ is a positive odd integer is similar, except we use
the obvious notion of a left-parent rather than a right-parent.

Finally, observe that each element of $Q$ appears at most twice when
reversing the incomparable pairs in $S(\Gamma)$. Recall that
$\cgL=\cgL'\cup\{L_{0},L_{0}'\}$, therefore
\[\ldim(Q)\leq \mu(Q,\cgL) \leq 2\cdot 2\cdot4^{2(t+1)}\cdot
5^{(t+1)^2}.
\] 
Now Lemma~\ref{lem:ldim-split} implies that
\[\ldim(P) \leq 2\cdot2\cdot 2\cdot4^{2(t+1)}\cdot 5^{(t+1)^2}+3.\]
Thus $\ldim(P)$ is $O(5^{(t+1)^2})$, as desired.

\subsection{Local Dimension and Tree-Width}\label{subsec:tw}

We now turn to proving Theorem~\ref{thm:ldim-tw}.  Recall that our
goal is to prove that for every $d\ge1$, there exists a poset $P$ such
that $\ldim(P)>d$ and the tree-width of the cover graph of $P$ is at
most~$3$.

Our argument will require some additional Ramsey theoretic tools.  The
results we use in the proof of Theorem~\ref{thm:ldim-tw} are treated
in a more comprehensive manner by
Milliken~\cite{bib:Mill}\footnote{The particular result we need is
  Theorem~2.1 on page 220. Note that Milliken credits the result to
  Halpern, L\"{a}uchi, Laver and Pincus and comments on the history of
  the result.}. However we will find it convenient to use somewhat
different notation and terminology.

For a positive integer $n$, we view the \textit{complete binary tree
  $T_n$} as the poset whose elements are the binary strings of length
at most~$n$, with $x\le y$ in $T_n$ when $x$ is a initial segment in
$y$.  The empty string, denoted $\emptyset$, is then the zero (least
element) of $T_n$. For all $n\ge1$, $T_n$ has $2^{n+1}-1$ elements,
$2^n$ leaves and height~$n+1$.  By convention, we take $T_0$ as the
one-point poset whose only element is the empty string.

When $n\ge1$ and $x$ is a binary string of length~$n$, we will denote
coordinate $i$ of $x$ as $x(i)$ and when a string is of modest length,
we may write it explicitly, e.g., $x=01001101$.  Let $x$ be a string
of length $p$, $y$ be a string of length~$m$ and $x<y$ in $T_n$. We
say $y$ is in the \textit{left tree above $x$} when
$y(p+1)=0$. Similarly, $y$ is in the \textit{right tree above $x$}
when $y(p+1)=1$.

Let $n$ and $m$ be integers with $n\ge m\ge0$, and let $\Lambda$ be a subposet of $T_n$. We will say $\Lambda$ is a \textit{strong copy of
  $T_m$} when there is a function $f:T_m\rightarrow \Lambda$ satisfying the
following two requirements:

\begin{enumerate}
\item $f$ is a poset isomorphism, i.e., $f$ is a bijection and for all
  $x,y\in T_m$, $x\le y$ in $T_m$ if and only if $f(x)\le f(y)$ in
  $\Lambda$.
\item For all $x,y\in T_m$ with $x<y$ in $T_m$, $y$ is in the left tree
  above $x$ in $T_m$ if and only if $f(y)$ is in the left tree above
  $f(x)$ in $T_n$.
\end{enumerate}

The following result is the special case of Theorem~2.1
from~\cite{bib:Mill} for binary trees and has also been used
in~\cite{bib:BFMMSTT}.  In fact, the application here preceded and
motivated the work in~\cite{bib:BFMMSTT}.

\begin{theorem}\label{thm:ramsey-bt}
  For every triple $(m,p,r)$ of positive integers, with $p\ge m$,
  there is a least positive integer $n_0=\BTRam(m,p,r)$ so that if
  $n\ge n_0$ and $\varphi$ is an $r$-coloring of the strong copies of
  $T_m$ in $T_n$, then there is a color $\alpha\in[r]$ and a subposet
  $\Lambda$ of $T_n$ such that $\Lambda$ is a strong copy of $T_p$ and
  $\varphi$ assigns color $\alpha$ to every strong copy of $T_m$
  contained in $\Lambda$.
\end{theorem}

We now turn our attention to a construction due to Joret, Micek and
Wiechert~\cite{bib:JoMiWi} which was used to show that a poset whose
cover graph has bounded tree-width can have dimension that grows
exponentially with the height of the poset.  Here is their
construction, with notation and terminology adjusted so that we can
conveniently apply Theorem~\ref{thm:ramsey-bt}.
 
For $n\ge0$, construct a poset $P_n$ as follows.  The ground set of
$P_n$ is the disjoint union $A_n\cup B_n$ where $B_n$ is an up set and
$A_{n}$ is a down set in $P_n$. The subposet $B_n$ is a copy of
$T_n$. For each string $x$ in $T_n$ of length at most $n$, we let
$b_x$ be the corresponding element of $B_n$, i.e., $b_x<b_y$ in $B_n$
if and only if $x$ is an initial segment of $y$. Note that
$b_\emptyset$ is the minimum of $B_n$.

The subposet $A_n$ is a copy of $T^*_n$, the dual of $T_n$. When
$x\in T_n$, we let $a_x$ be the corresponding element of $A_n$, i.e.,
$a_x>a_y$ in $A_n$ if and only if $x$ is an initial segment of
$y$. Note that $a_\emptyset$ is the maximum of $A_n$.

When $(a_x,b_y)\in A_n\times B_n$, we set $a_x<b_y$ if and only if
neither $x$ nor $y$ is an initial segment of the other.  For example,
$a_{1011}<b_{01}$, $a_{11}<b_{010}$, $a_{101}\parallel b_{10}$ and
$a_{101}\parallel b_{101}$ in $P_n$.

In Figure~\ref{fig:JMW} and Figure~\ref{fig:JMW2}, we show a
tree-decomposition of the cover graph of $P_1$ and $P_{2}$ that have
width~3 respectively. Observe that $P_1$ is isomorphic to the standard
example $S_3$. Also, note that there is a leaf of the host trees in
which the only two vertices of $P_1$ and $P_{2}$ which occur in this
bag are $a_\emptyset$ and $b_\emptyset$.
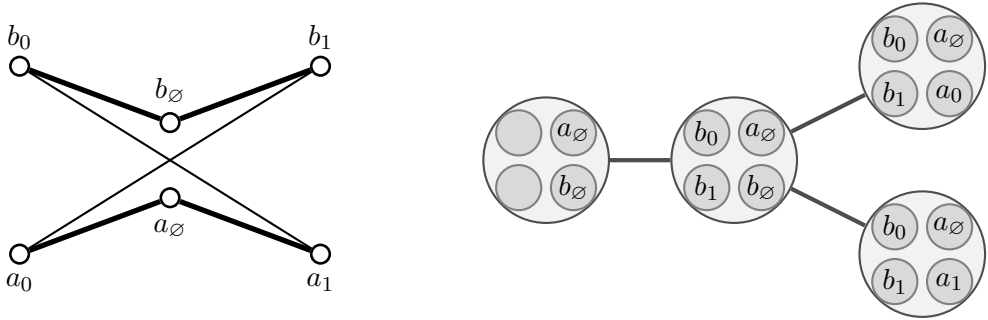
\begin{figure}
 \centering
 \input{fig2.tex}
 \caption{The Joret-Micek-Wiechert Construction for $n=1$}
 \label{fig:JMW}
\end{figure}

\begin{figure}
 \centering
 \input{fig3.tex}
 \caption{The Joret-Micek-Wiechert construction for $n=2$}
 \label{fig:JMW2}
\end{figure}
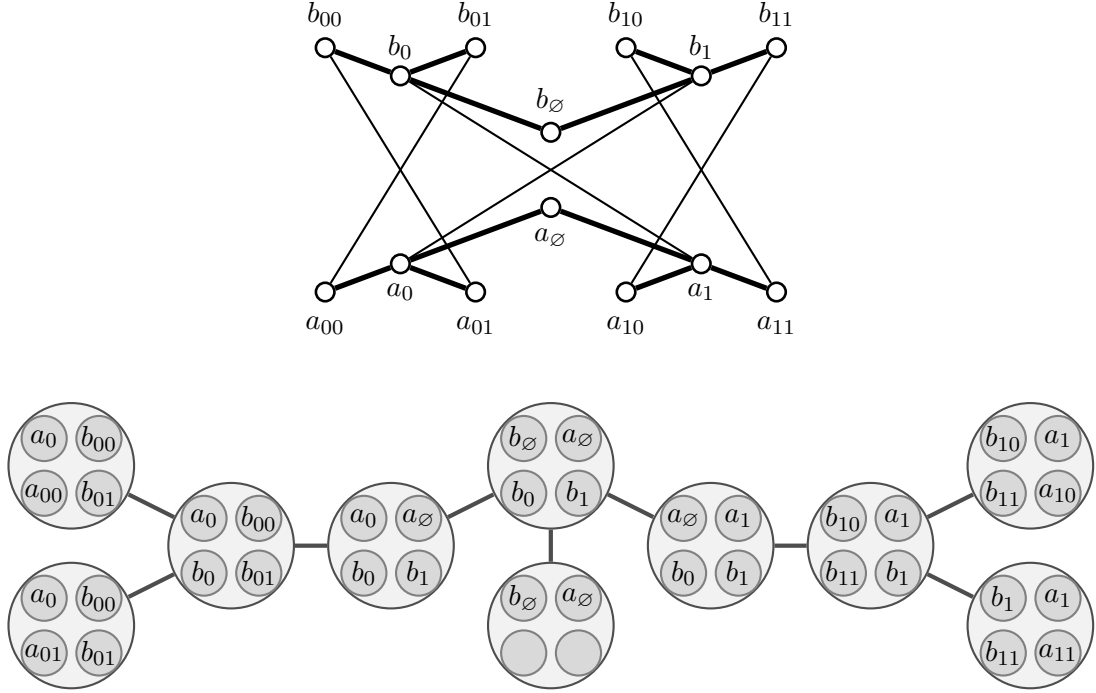

Now it is an easy exercise to verify by induction the following
properties of the family $\{P_n:n\ge0\}$:
\begin{enumerate}
\item The tree-width of the cover graph of $P_n$ is (at most)~$3$, and
\item $P_n$ has a tree-decomposition of width~$3$ for which there is a
  leaf (bag) $u$ in the host tree for which the set of elements of
  $P_n$ appearing in $u$ is precisely $\{a_\emptyset,b_\emptyset\}$.
\end{enumerate}
To complete the proof of Theorem~\ref{thm:ldim-tw}, we now prove the
following claim.

\medskip
\noindent
\textbf{Claim.}  Let $d\ge2$. If $n\ge\BTRam(1,3,d^2)$, then
$\ldim(P_n)>d$.

\begin{proof}
  Let $d\ge2$ and $n\ge\BTRam(1,3,d^2)$. We assume $\ldim(P_n)\le d$
  and argue to a contradiction. Let $\cgL=\{L_1,L_2,\dots,L_t\}$ be a
  local realizer for $P_n$ with $\mu(\cgL)\le d$. We use $\cgL$ to
  construct a coloring $\varphi$ of the strong copies of $T_1$ in
  $T_n$. A strong copy of $T_1$ consists of three binary strings
  $x,y,z$ with $x$ an initial segment of both $y$ and $z$ and, if $x$
  is a string of length $s$, $y(s+1)=0$ while $z(s+1)=1$. In
  particular, this implies that $a_x\parallel b_z$ in $P_n$.

  We then define a $d^2$-coloring of the strong copies of $T_1$ in
  $T_n$ by setting $\varphi(\{x,y,z\})=(\alpha,\beta)$ where $\alpha$
  and $\beta$ are defined as follows: Let $i$ be the least positive
  integer for which $a_x>b_z$ in $L_i$. Then $(\alpha,\beta)$ is the
  pair for which occurrence $\alpha$ of $a_x$ is in $L_i$ and
  occurrence $\beta$ of $b_z$ is in $L_i$. Since $\alpha,\beta\in[d]$,
  the function $\varphi$ uses at most $d^2$ colors. We pause to note
  that the element $y$ plays no role in this argument. Everything to
  follow depends only on $x$ and $z$.

  From Theorem~\ref{thm:ramsey-bt}, there is a subtree $\Lambda$ of
  $T_n$ and a color $(\alpha,\beta)$ so that $\Lambda$ is a strong
  copy of $T_3$ and $\varphi$ maps every strong copy of $T_1$ in
  $\Lambda$ to color $(\alpha,\beta)$.  Relabel the elements of
  $\Lambda$ so that they match the standard notation for $T_3$.

  Now consider the $3$-element subset $\{\emptyset,0,111\}$ in $T_3$,
  which is a strong copy of $T_1$.  This copy is assigned color
  $(\alpha,\beta)$ so there is some $L_i\in\cgL$ so that
  $a_\emptyset>b_{111}$ in $L_i$, where occurrence $\alpha$ of
  $a_\emptyset$ is in $L_i$ and occurrence $\beta$ of $b_{111}$ is in
  $L_i$. Next consider the $3$-element subset $\{\emptyset, 0, 101\}$
  which is also a strong copy of $T_1$. Since we already know that
  occurrence $\alpha$ of $a_\emptyset$ is in $L_i$, it follows that
  occurrence $\beta$ of $b_{101}$ is also in $L_i$.

  We then consider the subsets $\{10,100,101\}$ and $\{11,110,111\}$.
  Both are strong copies of $T_1$. Since we already know that
  occurrence $\beta$ of $b_{101}$ and occurrence $\beta$ of $b_{111}$
  is in $L_i$, we conclude that $a_{10}>b_{101}$ in $L_i$ and
  $a_{11}>b_{111}$ in $L_i$. This is impossible since $a_{10}<b_{111}$
  in $P_n$ and $a_{11}<b_{101}$ in $P_n$.  The contradiction completes
  the proof.
\end{proof}

\section{Summary Listing of Open Problems}

For the convenience of readers, we gather here a listing of open
problems concerning Boolean dimension and local dimension.

\begin{enumerate}
\item For a positive integer $w$, what is the maximum value of the
  Boolean dimension of a poset whose width is~$w$?
\item For a positive integer $w$, what is the maximum value of the
  local dimension of a poset whose width is~$w$?
\item For a non-negative integer $n$, what is the maximum value of the
  Boolean dimension of a poset consisting of an antichain and $n$
  additional points?
\item Is there a constant $d_0$ such that every planar poset has
  Boolean dimension at most $d_0$?
\item Is there a constant $d_0$ such that every poset with a planar
  cover graph has Boolean dimension at most $d_0$?
\item If a planar poset has large dimension, must it contain a large
  standard example?
\item If a planar poset has large Boolean dimension, must it contain a
  large standard example?
\item If a planar poset has large local dimension, must it contain a large
standard example?
\item For an integer $d\ge4$, what is the maximum local dimension of a
  disconnected poset in which each component has local dimension at
  most~$d$?  \textit{Note.} The answer is either $d$, $d+1$ or $d+2$.
\item What is the maximum amount the Boolean dimension of a poset can
  drop when a single point is removed?  \textit{Note.} The answer is
  either $1$, $2$ or $3$.
\item What is the Boolean dimension and the local dimension of
  $\mathbf{2}^d$?
\end{enumerate}

\section{Acknowledgments}

Our work has benefited considerably through collaboration, and a touch
of competition, with our colleagues Stefan Felsner, Gwena\"{e}l Joret,
Tam\'{a}s M\'{e}sz\'{a}ros, Piotr Micek and Bartosz Walczak. Smith was
supported in part by NSF-DMS grant 1344199.

\end{document}

%% file: fig1.tex
\begin{tikzpicture}[scale=0.9]

\begin{scope}[nodes={circle,fill=white,draw=black,line width=1pt,inner sep=0pt, minimum size=7pt}]
\foreach \i in {1,2,3,4}{
	\node (x\i) at (2*\i,0) {};
	\node (y\i) at (2*\i,2) {};
	\node (z\i) at (2*\i,4) {};
}
\node[label={[label distance=2.5mm]below:$a_1$}] (a1) at (5,6) {};
\node (a2) at (10,2) {};
\end{scope}

\node[label=above:$a_{2}$] at (a2) {};

\node[label=below:$x_{1}$] at (x1) {};
\node[label=below:$x_{2}$] at (x2) {};
\node[label=below:$x_{3}$] at (x3) {};
\node[label=below:$x_{4}$] at (x4) {};

\node[label=above:$y_{1}$] at (y1) {};
\node[label=above:$y_{2}$] at (y2) {};
\node[label=above:$y_{3}$] at (y3) {};
\node[label=above:$y_{w}$] at (y4) {};

\node[label=below:$z_{1}$] at (z1) {};
\node[label=below:$z_{2}$] at (z2) {};
\node[label=below:$z_{3}$] at (z3) {};
\node[label=below:$z_{w}$] at (z4) {};

\path[name=xx] (x3) -- (x4) node [midway] {\Large$\cdots$};
\path[name=yy] (y3) -- (y4) node [midway] {\Large$\cdots$};
\path[name=zz] (z3) -- (z4) node [midway] {\Large$\cdots$};

\begin{scope}[line width=1pt]
\foreach \i in {1,2,3,4}{
\draw (x\i) edge (y\i);
\draw (x\i) edge (a2);
\draw (z\i) edge (a1);
	\foreach \j in {1,2,3,4}{
		\ifthenelse{\NOT \i = \j}{\draw (y\i) edge (z\j);}{}
	}
}
\end{scope}

\node (P) at (5,-1) {\Large $P_{w}$};

\end{tikzpicture}

%% file: kelly.tex
\begin{tikzpicture}[nodes={circle,fill=white,draw=black,line width=1pt,inner sep=0pt, minimum size=7pt},scale=0.38]

\begin{scope}[yshift=10mm,xshift=5mm]
		\node[label={above:$w_1$}]  (0) at (-8, 0) {};
		\node[label={above:$w_2$}]  (1) at (-6, 1) {};
		\node[label={above:$w_3$}]  (2) at (-4, 2) {};
		\node[label={above:$w_4$}]  (3) at (-2, 3) {};
\end{scope}

\begin{scope}[yshift=-10mm,xshift=-5mm]
		\node[label={below:$z_1$}]  (6) at (-12, -2) {};
		\node[label={below:$z_2$}]  (7) at (-14, -3) {};
		\node[label={below:$z_3$}]  (8) at (-16, -4) {};
		\node[label={below:$z_4$}]  (9) at (-18, -5) {};
\end{scope}

\begin{scope}[xshift=-5mm]
		\node[label={north east:$b_1$}]  (5) at (-11, 1) {};
		\node[label={north east:$b_2$}]  (12) at (-12, 3) {};
		\node[label={north east:$b_3$}]  (13) at (-13, 5) {};
		\node[label={north east:$b_4$}]  (14) at (-14, 7) {};
		\node[label={north east:$b_5$}]  (17) at (-15, 9) {};
\end{scope}

		\node[label={south west:$a_1$}]  (4) at (-9, -3) {};
		\node[label={south west:$a_2$}]  (10) at (-8, -5) {};
		\node[label={south west:$a_3$}]  (11) at (-7, -7) {};
		\node[label={south west:$a_4$}]  (15) at (-6, -9) {};
		\node[label={south west:$a_5$}]  (16) at (-5, -11) {};

\begin{scope}[line width=1pt]
		\draw (0) edge (1);
		\draw (1) to (2);
		\draw (2) to (3);
		\draw (6) to (7);
		\draw (7) to (8);
		\draw (8) to (9);
		\draw (4) to (0);
		\draw (0) to (12);
		\draw (10) to (1);
		\draw (1) to (13);
		\draw (11) to (2);
		\draw (2) to (14);
		\draw (5) to (6);
		\draw (6) to (10);
		\draw (12) to (7);
		\draw (7) to (11);
		\draw (15) to (3);
		\draw (8) to (15);
		\draw (9) to (16);
		\draw (3) to (17);
		\draw (13) to (8);
		\draw (14) to (9);
\end{scope}

\end{tikzpicture}

%% file: fig2.tex
\begin{tikzpicture}

\begin{scope}[nodes={circle,fill=white,draw=black,line width=1pt,inner sep=0pt, minimum size=7pt}]

\begin{scope}[yscale=-1, yshift=0.5cm]
\node[label={below:$a_{\emptyset}$}] (a) at (0,0) {};
\node[label={below:$a_{1}$}] (a1) at (2,.75) {};
\node[label={below:$a_{0}$}] (a0) at (-2,.75) {};
\end{scope}
\begin{scope}[yshift=0.5cm]
\node[label={above:$b_{\emptyset}$}] (b) at (0,0) {};
\node[label={above:$b_{1}$}] (b1) at (2,.75) {};
\node[label={above:$b_{0}$}] (b0) at (-2,.75) {};
\end{scope}
\end{scope}

\begin{scope}[line width=0.8pt]
\begin{scope}[line width=2pt]
\draw (a) edge (a1);
\draw (a) edge (a0);
\draw (b) edge (b1);
\draw (b) edge (b0);
\end{scope}
\draw (a1)edge(b0);
\draw (a0)edge(b1);
\end{scope}

\begin{scope}[inner/.style={circle,draw=black!50,fill=gray!30,thick,inner sep=0.5pt,minimum size=6mm},
  outer/.style={circle,draw=black!70,fill=gray!10,thick,inner sep=-2.5pt},xshift=5cm]
\node[outer] (T) at (5,1.25) {
\begin{tikzpicture}[node distance=0.1cm]
\node[inner] (t1) {$b_{0}$};
\node[inner,right=of t1] (t2) {$a_{\emptyset}$};
\node[inner,below=of t1] (t3) {$b_{1}$};
\node[inner,right=of t3] (t4) {$a_{0}$};
\end{tikzpicture}
};

\node[outer] (S) at (2.5,0) {
\begin{tikzpicture}[node distance=0.1cm]
\node[inner] (s1) {$b_{0}$};
\node[inner,right=of s1] (s2) {$a_{\emptyset}$};
\node[inner,below=of s1] (s3) {$b_{1}$};
\node[inner,right=of s3] (s4) {$b_{\emptyset}$};
\end{tikzpicture}
};

\node[outer] (R) at (5,-1.25) {
\begin{tikzpicture}[node distance=0.1cm]
\node[inner] (r1) {$b_{0}$};
\node[inner,right=of r1] (r2) {$a_{\emptyset}$};
\node[inner,below=of r1] (r3) {$b_{1}$};
\node[inner,right=of r3] (r4) {$a_{1}$};
\end{tikzpicture}
};

\node[outer] (Q) at (0,0) {
\begin{tikzpicture}[node distance=0.1cm]
\node[inner] (q1) {};
\node[inner,right=of q1] (q2) {$a_{\emptyset}$};
\node[inner,below=of q1] (q3) {};
\node[inner,right=of q3] (q4) {$b_{\emptyset}$};
\end{tikzpicture}
};

\draw[ultra thick,color=black!70] (S) edge (T);
\draw[ultra thick,color=black!70] (S) edge (R);
\draw[ultra thick,color=black!70] (S) edge (Q);

\end{scope}

\end{tikzpicture}

%% file: fig3.tex
\begin{tikzpicture}

\begin{scope}[nodes={circle,fill=white,draw=black,line width=1pt,inner sep=0pt, minimum size=7pt},yshift=5cm]
	\begin{scope}[yscale=-1, yshift=0.5cm]
		\node[label={below:$a_{\emptyset}$}] (a) at (0,0) {};
		\node[label={below:$a_{1}$}] (a1) at (2,.75) {};
		\node[label={below:$a_{10}$}] (a10) at (0.5*2,1.5*.75) {};
		\node[label={below:$a_{11}$}] (a11) at (1.5*2,1.5*.75) {};
		\node[label={below:$a_{0}$}] (a0) at (-2,.75) {};
		\node[label={below:$a_{00}$}] (a00) at (-2*1.5,1.5*0.75) {};
		\node[label={below:$a_{01}$}] (a01) at (-2*0.5,1.5*0.75) {};
	\end{scope}
	\begin{scope}[yshift=0.5cm]
		\node[label={above:$b_{\emptyset}$}] (b) at (0,0) {};
		\node[label={above:$b_{1}$}] (b1) at (2,.75) {};
		\node[label={above:$b_{10}$}] (b10) at (0.5*2,1.5*.75) {};
		\node[label={above:$b_{11}$}] (b11) at (1.5*2,1.5*.75) {};
		\node[label={above:$b_{0}$}] (b0) at (-2,.75) {};
		\node[label={above:$b_{00}$}] (b00) at (-2*1.5,1.5*0.75) {};
		\node[label={above:$b_{01}$}] (b01) at (-2*0.5,1.5*0.75) {};
	\end{scope}
\end{scope}

\begin{scope}[line width=0.8pt]
	\begin{scope}[line width=2pt]
		\draw (a) edge (a1);
		\draw (a) edge (a0);
		\draw (a1) edge (a11);
		\draw (a1) edge (a10);
		\draw (a0) edge (a01);
		\draw (a0) edge (a00);
		\draw (b) edge (b1);
		\draw (b) edge (b0);
		\draw (b1) edge (b11);
		\draw (b1) edge (b10);
		\draw (b0) edge (b01);
		\draw (b0) edge (b00);
	\end{scope}
	\draw (a1)edge(b0);
	\draw (a0)edge(b1);
	\draw (a11)edge(b10);
	\draw (a10)edge(b11);
	\draw (a01)edge(b00);
	\draw (a00)edge(b01);
\end{scope}

\begin{scope}[scale=0.85,inner/.style={circle,draw=black!50,fill=gray!30,thick,inner sep=0.2pt,minimum size=6mm},
  outer/.style={circle,draw=black!70,fill=gray!10,thick,inner sep=-2.5pt},xshift=-7.5cm]

\begin{scope}[xshift=10cm]
\node[outer] (T1) at (5,1.25) {
\begin{tikzpicture}[node distance=0.1cm]
\node[inner] (t1) {$b_{10}$};
\node[inner,right=of t1] (t2) {$a_{1}$};
\node[inner,below=of t1] (t3) {$b_{11}$};
\node[inner,right=of t3] (t4) {$a_{10}$};
\end{tikzpicture}
};

\node[outer] (S1) at (2.5,0) {
\begin{tikzpicture}[node distance=0.1cm]
\node[inner] (s1) {$b_{10}$};
\node[inner,right=of s1] (s2) {$a_{1}$};
\node[inner,below=of s1] (s3) {$b_{11}$};
\node[inner,right=of s3] (s4) {$b_{1}$};
\end{tikzpicture}
};

\node[outer] (R1) at (5,-1.25) {
\begin{tikzpicture}[node distance=0.1cm]
\node[inner] (r1) {$b_{1}$};
\node[inner,right=of r1] (r2) {$a_{1}$};
\node[inner,below=of r1] (r3) {$b_{11}$};
\node[inner,right=of r3] (r4) {$a_{11}$};
\end{tikzpicture}
};

\node[outer] (Q1) at (0,0) {
\begin{tikzpicture}[node distance=0.1cm]
\node[inner] (q1) {$a_{\emptyset}$};
\node[inner,right=of q1] (q2) {$a_{1}$};
\node[inner,below=of q1] (q3) {$b_{0}$};
\node[inner,right=of q3] (q4) {$b_{1}$};
\end{tikzpicture}
};

\draw[ultra thick,color=black!70] (S1) edge (T1);
\draw[ultra thick,color=black!70] (S1) edge (R1);
\draw[ultra thick,color=black!70] (S1) edge (Q1);
\end{scope}

\begin{scope}[xscale=-1,xshift=-5cm]
\node[outer] (T0) at (5,1.25) {
\begin{tikzpicture}[node distance=0.1cm]
\node[inner] (t1) {$b_{00}$};
\node[inner,left=of t1] (t2) {$a_{0}$};
\node[inner,below=of t1] (t3) {$b_{01}$};
\node[inner,left=of t3] (t4) {$a_{00}$};
\end{tikzpicture}
};

\node[outer] (S0) at (2.5,0) {
\begin{tikzpicture}[node distance=0.1cm]
\node[inner] (s1) {$b_{00}$};
\node[inner,left=of s1] (s2) {$a_{0}$};
\node[inner,below=of s1] (s3) {$b_{01}$};
\node[inner,left=of s3] (s4) {$b_{0}$};
\end{tikzpicture}
};

\node[outer] (R0) at (5,-1.25) {
\begin{tikzpicture}[node distance=0.1cm]
\node[inner] (r1) {$b_{00}$};
\node[inner,left=of r1] (r2) {$a_{0}$};
\node[inner,below=of r1] (r3) {$b_{01}$};
\node[inner,left=of r3] (r4) {$a_{01}$};
\end{tikzpicture}
};

\node[outer] (Q0) at (0,0) {
\begin{tikzpicture}[node distance=0.1cm]
\node[inner] (q1) {$a_{\emptyset}$};
\node[inner,left=of q1] (q2) {$a_{0}$};
\node[inner,below=of q1] (q3) {$b_1$};
\node[inner,left=of q3] (q4) {$b_{0}$};
\end{tikzpicture}
};

\draw[ultra thick,color=black!70] (S0) edge (T0);
\draw[ultra thick,color=black!70] (S0) edge (R0);
\draw[ultra thick,color=black!70] (S0) edge (Q0);
\end{scope}

\node[outer] (T) at (7.5,1.25) {
\begin{tikzpicture}[node distance=0.1cm]
\node[inner] (t1) {$b_{\emptyset}$};
\node[inner,right=of t1] (t2) {$a_{\emptyset}$};
\node[inner,below=of t1] (t3) {$b_{0}$};
\node[inner,right=of t3] (t4) {$b_{1}$};
\end{tikzpicture}
};

\node[outer] (Q) at (7.5,-1.25) {
\begin{tikzpicture}[node distance=0.1cm]
\node[inner] (t1) {$b_{\emptyset}$};
\node[inner,right=of t1] (t2) {$a_{\emptyset}$};
\node[inner,below=of t1] (t3) {};
\node[inner,right=of t3] (t4) {};
\end{tikzpicture}
};

\draw[ultra thick,color=black!70] (Q0) edge (T);
\draw[ultra thick,color=black!70] (Q1) edge (T);
\draw[ultra thick,color=black!70] (Q) edge (T);
\end{scope}

\end{tikzpicture}